\documentclass[12pt]{amsart} \usepackage{fullpage}
\usepackage[T1]{fontenc}
\usepackage[utf8]{inputenc}

\usepackage{amsfonts,amsmath,amsthm,amssymb}
\usepackage{enumitem} 
\usepackage{siunitx} % spacing in (big) numbers
\usepackage[hidelinks,linktoc=all]{hyperref}

\usepackage{pgf} 
\usepackage{graphicx}
\usepackage{subcaption}
\usepackage{adjustbox}
\newcommand{\superficial}[1]{{\mathcal{O}_{#1}}} %superficial semigroup with multiplicity
\DeclareMathOperator{\wilfoper}{W} %Wilf number
\DeclareMathOperator{\eliahouoper}{E} %Eliahou number
\DeclareMathOperator{\Frobeniusoper}{F} %Frobenius
\DeclareMathOperator{\multiplicityoper}{m} %multiplicity
\DeclareMathOperator{\conductoroper}{\chi} %conductor -- notation used by Wilf 
\DeclareMathOperator{\genusoper}{\Omega} %genus -- notation used by Wilf 
\DeclareMathOperator{\depthoper}{h} %q
\DeclareMathOperator{\rhooper}{\rho} %rho
\DeclareMathOperator{\Ddepthoper}{D_{\depthoper}} %rho
\DeclareMathOperator{\densityoper}{d} %d
\DeclareMathOperator{\leftsoper}{L} %lefts
\DeclareMathOperator{\primitivesoper}{P} %primitives
\DeclareMathOperator{\leftprimitivesoper}{LP} %left primitives
\DeclareMathOperator{\bigprimitivesoper}{BP} %big primitives
\DeclareMathOperator{\embeddingdimensionoper}{e} %number of primitives %embeding dimension
\DeclareMathOperator{\leftembeddingdimensionoper}{le} %number of left primitives %left embeding dimension

\newcommand{\prop}[2]{\mathcal{#1}_{#2}}

\newtheorem{theorem}{Theorem}[section]
\newtheorem{lemma}[theorem]{Lemma}
\newtheorem{corollary}[theorem]{Corollary}
\newtheorem{proposition}[theorem]{Proposition}
\newtheorem{conjecture}[theorem]{Conjecture}

\theoremstyle{remark} 
\newtheorem{example}[theorem]{Example}
\newtheorem{remark}[theorem]{Remark}
\newtheorem{question}[theorem]{Question}
\newtheorem{problem}[theorem]{Problem}
\newtheorem{definition}[theorem]{Definition}
\title{Trimming the numerical semigroups tree to probe Wilf's conjecture to higher genus}
\author{Manuel Delgado}
\date{\today}
\date{\today}
\address{CMUP, Departamento de Matem\'atica, Faculdade de
	Ci\^encias, Universidade do Porto, Rua do Campo Alegre 687, 4169-007 Porto,
	Portugal} 
\email{mdelgado@fc.up.pt} 
\thanks{The author was partially supported by CMUP (UID/MAT/00144/2019), which is funded by FCT (Portugal) with national (MCTES) and European structural funds (FEDER), under the partnership agreement PT2020, and also by the Spanish project MTM2017-84890-P. Furthermore, the author acknowledges a sabbatical grant from the FCT: SFRH/BSAB/142918/2018.}
\begin{document}


% 20M14 Commutative semigroups
% 20--02 Group theory and generalizations -- survey articles
% 05--02 Combinatorics -- survey articles
% 11--02 Number theory -- survey articles
% 20--04 Explicit machine computation and programs (not the theory of computation or programming) 
% 05A Enumerative combinatorics [For enumeration in graph theory, see 05C30] 
% 05A15 Exact enumeration problems, generating functions [See also 33Cxx, 33Dxx]

\begin{abstract}%150 words maximum (EM) %300 words maximum (CM)
  This paper aims to contribute to validate, for numerical semigroups of reasonably large
  genus, the so-called Conjecture of Wilf. There is no counter-example for the conjecture among the over \(3\cdot10^{10}\) numerical semigroups of genus up to \(60\), as it has been computationally verified by Fromentin and Hivert.
  The computations use the idea of parsing a semigroups tree, making tests in each node. As a mean to combine parsing of the semigroups tree with known theoretical results, we introduce the concept of 
  \emph{cutting semigroup}.
  
  Assume that there exists a property on numerical semigroups that implies that a semigroup satisfying it necessarily satisfies Wilf's conjecture. Assume also that, besides, this property is hereditary, that is, if a semigroup satisfies it, then all its descendants in the semigroups tree also have the same property. Such properties exist. Interesting ones can be easily deduced from some deep results of Eliahou.
  A semigroup satisfying such a property is called a cutting semigroup (for that property).
  When looking for counter-examples to Wilf's conjecture, if a cutting semigroup is found, then it can be cut off, since no counter example lies among its descendants. One can therefore consider the trimmed tree obtained by cutting off all the cutting semigroups.
  
  Depending on the properties, these ideas allow in some cases to avoid parsing a very significant part of the tree and can be applied to other related problems. 
  For instance, they can be used to find all numerical semigroups with negative Eliahou number up to a genus bigger than \(60\), which is the current record. Computations by Fromentin had shown that there are exactly \(5\) numerical semigroups in these circumstances.
\end{abstract}

\maketitle

%%%%%%%%%%%%%%%%
\tableofcontents
%%%%%%%%%%%%%%%%

\newpage

%%%%%%%%%%% introduction %%%%%%%%%%
\section{Introduction}\label{sec:introduction}
%%\section{Introduction}
A numerical semigroup is a co-finite submonoid of the additive monoid of nonnegative integers \(\mathbb{N}\). Most of the general facts simply referred as \emph{well-known} in this text are proved (or conveniently referenced) in a book by Rosales and García-Sánchez~\cite{RosalesGarcia2009Book-Numerical}.  For instance, it is well known that a numerical semigroup has a unique minimal set of generators, which is finite. For an instant denote its cardinality by \(k\) (later, the commonly used nomenclature embedding dimension will be introduced). 
The positive integers not belonging to a numerical semigroup are its gaps; for the moment call them omitted values. The number of omitted values (later called genus) is denoted by \(\genusoper\).
Since a numerical semigroup \(S\) is a co-finite subset of the set of nonnegative integers, there exists the greatest integer that does not belong to \(S\). It is called the Frobenius number of \(S\), and its successor (which happens to be the least element of \(S\) such that all the integers greater than it belong to \(S\)) is the conductor of the semigroup. It is denoted \(\conductoroper\).

This paper is part of a bigger project involving experiments with numerical semigroups, most of them motivated by questions raised by Wilf in 1978. In this section we recall those questions, explain where the theory developed here fits into the project and give an outline of the present paper. 

%%%%%%%%
\subsection{Questions raised by Wilf}\label{subsec:Wilf_questions} 

In the last paragraph of~\cite{Wilf1978AMM-circle}, Wilf posed  two problems:
    \begin{itemize}
    \item[(a)] Is it true that for a fixed \(k\) the fraction \(\Omega/\chi\) of
      omitted values is at most \(1-(1/k)\) with equality only for the
      generators \(k,k+1,\ldots,2k-1\)?
    \item[(b)] Let \(f(n)\) be the number of semigroups whose conductor is
      \(n\). What is the order of magnitude of \(f(n)\) for \(n\to\infty\)?
	 \end{itemize}

These problems attracted the attention of many researchers, with an outcome of several high quality papers.  
A few comments on both problems follow.

%%%%%%%%
\subsubsection*{Comments on Problem (a)}

    The first part of Problem (a) is nowadays known as Wilf's conjecture. A result of Sylvester gives counter-examples for the second part (see Remark~\ref{rem:Silvester_implies_Wilf}); apparently Wilf forgot it. 
	Recently I wrote a survey on Wilf's conjecture (see~\cite{Delgado2019ae-Conjecture}) which, in particular, recalls many of the results one can find in the literature.   
  
    As any problem that attracts the interest of many researchers during a considerable amount of time, Problem (a) gives rise to other related problems which are interesting themselves. One of these problems was posed by Eliahou~\cite{Eliahou2018JEMS-Wilfs}, is object of active research (see~\cite{Delgado2018MZ-question,EliahouFromentin2019SF-misses}) and is also one of the motivating sources for the present work.% (see Section~\ref{sec:eliahou}). 
%%%%%%%%%%

	Researchers are nowadays convinced that the equality referred by Wilf occurs only for two specific families of semigroups (see~\cite[Quest.~8]{MoscarielloSammartano2015MZ-conjecture}), but, to the best of my knowledge, it has never been stated as a conjecture. In Section~\ref{subsec:0-wilf}, we will refer it as the \(0\)-Wilf problem. 
	
%%%%%%%%%%
\subsubsection*{Comments on Problem (b)}
	The second of the problems posed by Wilf and stated above, Problem (b), is on counting numerical semigroups. One could say that it is a problem on counting by conductor (or, equivalently, by Frobenius number). For instance, Backelin~\cite{Backelin1990MS-number} has considered it. In~\cite{RosalesGarcia-SanchezGarcia-GarciaJimenezMadrid2004JPAA-Fundamental} one can find the number of numerical semigroups with Frobenius number up to \(39\).
	Due to some conjectures raised by Bras-Amorós~\cite{Bras-Amoros2008SF-Fibonacci}, the main interest moved to counting numerical semigroups by genus. As will be referred later, there are satisfactory answers to some of the conjectures by Bras-Amorós. Kaplan wrote a survey on the theme that is worth to read (see~\cite{Kaplan2017AMM-Counting}).
	
	Possibly motivated by Bras-Amorós works~\cite{Bras-AmorosBulygin2009SF-Towards,Bras-Amoros2009JPAA-Bounds} (see also~\cite{Elizalde2010JPAA-Improved}), the computational approaches to this problem (see~\cite{FromentinHivert2016MC-Exploring,Bras-AmorosFernandez-Gonzalez2018MC-Computation}) are based on exploring a certain tree of numerical semigroups. We do not escape from it.
	
%%%%%%%%%%
\subsection{This paper as part of an experimental project}\label{subsec:software_experiments} 
	Most of the theory related to a bigger experimental project is developed in the present paper. It roughly consists of putting together several rather simple ideas. Many of them are either directly suggested by the computations meanwhile done or appeared as a result of research that the computations revealed as necessary.
  	
  	The focus was on probing some conjectures on numerical semigroups up to some genus,
  	as well as finding examples of numerical semigroups satisfying very special properties. This increases the knowledge on the semigroups satisfying those properties and can ultimately contribute to characterize the families they define.
  	Many researchers have spent some time dealing with these kinds of problems, and various experimental results appear in the literature. The records stay presently around genus \(60\) or \(70\) (see, in particular, Section~\ref{subsec:counting_by_genus}). 
	
	The starting question for this project can be stated as ``How to integrate theoretical (possibly deep) results to reduce the search space for the examples?''. Having a satisfactory answer to this question, the will to break the records known appears naturally and that constitutes the more computational and experimental aspect of this project. The development of some theory and the different aspects just referred to have been developed simultaneously, with mutual benefits to all of them.
  	
  	The development of some theory and the alignment of some ideas is presented here, as referred. Other aspects are to appear in different works. A summary follows.
  	 	
  	One of the aspects is the development of some software in \textsf{GAP}~\cite{GAP4-2019}. Observe that the computational system \textsf{GAP} has its own high-level programming language (which is an interpreted language, not surprisingly called \textsf{GAP}). One should note that using a high-level programming language has big advantages on the production (and verification) of code, but one can not expect it to be fast, when compared with highly optimized programs produced using lower-level code.  	 
  	The code produced will be available in the form of a \textsf{GAP}~\cite{GAP4-2019} package~\cite{NStree-2019}, following all the standards of a \textsf{GAP} package, namely in what concerns the documentation. It consists of pure \textsf{GAP} code, and is to be proposed to the \textsf{GAP} Council for distribution as a standard \textsf{GAP} package. The development version of this open-source software will be available in GitHub.
  	
  	Another aspect is of experimental nature~\cite{Delgado2019-Probing}. It consists mainly of experiments made using the software developed. In particular, many of the tables produced are presented there. They point to an asymptotic behaviour of some of the sequences of numbers obtained, which leads to questions that may be worth to consider.
  	  	
  	One question that the reader naturally poses is whether the ideas presented may lead to performing tests for numbers significantly bigger than the already known ones (without increasing the time or computational means).
  	Here is where another aspect of the project comes into play: produce and use fast implementations. This was achieved with Jean Fromentin~\cite{DelgadoFromentin2019}, and some of the previous records were pulverized.  
  	
%%%%%%%%%%
\subsection{Structure of the paper}\label{subsec:structure_of_paper}
	Besides this introduction where motivating questions due to Wilf are stated and some explanation on where this paper fits into a bigger project is given, the paper is divided into several sections whose content is summarized in what follows. 
	
  	\medskip
    
    In Section~\ref{sec:generalities}, some concepts and general notation that will be used along the paper, are introduced.
    Among these concepts is \emph{the} semigroup tree.
  	
   	Then, in Section~\ref{sec:counting_search_space}, we identify a search space for (counter) examples. It consists of the numerical semigroups with genus up to a given number and is tightly related to counting numerical semigroups by genus. A tree of numerical semigroups is commonly used to that effect; the search space is obtained by truncating the tree at a certain level.
   
   	Motivating problems, most of them well known to the researchers in the area, are stated in Section~\ref{sec:motivating}; then computational challenges associated to those problems are drawn.
   
   	Some results on semigroups generated by the elements to the left of the conductor appear in Section~\ref{sec:lefts}. These results are then used in Section~\ref{sec:trimming_global}, which is the heart of the paper. That section is about the integration of theoretical results in a process of trimming the numerical semigroups tree as a way of tackling the computational challenges previously proposed.
   
   	Section~\ref{sec:trimming_truncated} is also dedicated to trimming, but this time truncated trees. In any case, the trimming is made according to the challenge in mind.
   
   	A few technical details are presented in Section~\ref{sec:technical_details}, the last in the paper.
  
  	\medskip

	Some of the examples presented involve computations with numerical semigroups. The examples have been chosen so that the reader can easily do the verifications by hand. In some cases, \textsf{GAP} sessions are included to illustrate them. We encourage the readers who have at their disposal a working version of the \textsf{numericalsgps}~\cite{NumericalSgps1.2.0} \textsf{GAP}~\cite{GAP4-2019} package to adapt the commands used in these \textsf{GAP} sessions to produce their own examples possibly involving bigger numbers.

	From time to time I express some opinions which I see as being of a rather personal nature. In those cases I avoid using the colloquial \emph{we} which is generally adopted along the paper.

%%\newpage
%%%%%%%%%%% generalities %%%%%%%%%%
\section{Generalities}\label{sec:generalities}
%%\section{Generalities}
The letter \(S\) is used to denote a numerical semigroup. We often follow the quite common simplifying practice of writing \(S\) instead of explicitly saying that \(S\) is a numerical semigroup. 

The sets of integers considered in this paper are endowed with the usual order and thought as being discrete sets in an horizontal line. For integers \(a\) and \(b\), sentences like ``\(a\) is to the right of \(b\)'' will be freely used with the usual meaning ``\(a > b\)''.   
 
Notation and terminology used in various parts of the paper are introduced below. Most of it is commonly used in the literature. A few well-known remarks are recalled.
Some more notation appears later, as the necessary concepts are introduced. 

The second part of this section is devoted to \emph{the} numerical semigroups tree. A construction is recalled and a few observations on its exploration are made.

%%%%%%%%%
\subsection{Common notation and terminology}\label{sec:general_notation}   
 	For a set \(X\) of positive integers, \(\langle X\rangle\) denotes the submonoid of \(\mathbb{N}\) generated by \(X\). It is well known that \(\langle X\rangle\) is a numerical semigroup if and only if the elements of \(X\) are globally coprime, that is \(\gcd(X)=1\). 
 	The notation \(\langle X\rangle_t\) is used to represent the smallest numerical semigroup that contains \(X\) and all the integers greater than or equal to~\(t\). We will refer to it as the semigroup generated by \(X\) \emph{truncated} at~\(t\). 
 	As usual, \(\lvert X \rvert\) denotes the cardinality of a set \(X\).
 	Occasionally (integer) interval notation will be used: for instance, \(\left[a,b\right]\) is used for  \(\left[a,b\right]\cap \mathbb{Z}\). 
 	For a given integer \(n\), the infinite interval \(\{x\in \mathbb{Z}\mid x\ge n\}\) is denoted \(\left\{n,\rightarrow\right\}\) instead, as is common in the area.
 
  	Recall that a numerical semigroup \(S\) has a unique minimal set of generators, which is finite. Its elements are the \emph{minimal generators} of~\(S\), also known as \emph{primitives} of~\(S\) (the latter is the terminology adopted in the present paper).
 	The set of primitives of~\(S\) is denoted \(\primitivesoper(S)\), or simply \(\primitivesoper\), when there is no possible confusion. A non primitive element of \(S\setminus\{0\}\) is called \emph{decomposable}.
 	\begin{remark}\label{rem:primitive_in_subsemigroup}  
 		Let \(S\) and \(S'\) be numerical semigroups with \(S'\subset S\) and let \(p\in S'\cap \primitivesoper(S)\). Then \(p\in \primitivesoper(S')\).
 	\end{remark}
 	\begin{proof}
 		Suppose that \(p\) is decomposable in \(S'\). Then it is decomposable in \(S\) (with the same decomposition).
 	\end{proof}
 	The cardinality of \(\primitivesoper(S)\) is said to be the \emph{embedding dimension of~\(S\)}. It is denoted \(\embeddingdimensionoper(S)\). The notation is frequently simplified to~\(\embeddingdimensionoper\), as long as there is no possible confusion.
  
 	Recall that the \emph{conductor of~\(S\)} is the smallest element of~\(S\) such that all integers to its right belong to~\(S\). It is denoted by \(\conductoroper(S)\), or simply by~\(\conductoroper\). The \emph{Frobenius number} of \(S\) is \(\conductoroper(S)-1\) and is denoted \(\Frobeniusoper(S)\).
 
 	The \emph{multiplicity of~\(S\)} is the least positive element of~\(S\). It is denoted by \(\multiplicityoper(S)\), or simply~\(\multiplicityoper\).
 	It is well known that \(\embeddingdimensionoper(S)\le \multiplicityoper(S)\); when equality occurs, \(S\) is said to be of \emph{maximal embedding dimension}.
 
 	The \emph{depth of~\(S\)} is the integer \(\depthoper(S)=\lceil \conductoroper(S)/\multiplicityoper(S)\rceil\). It gives a measure of how large is the multiplicity of a numerical semigroup when compared to its conductor (or, in other words, how little is the conductor when compared to the multiplicity). This important parameter appears, for instance, in~\cite{EliahouFromentin2020JCTSA-Gapsets}. Another parameter associated to a numerical semigroup that can be defined similarly is \(\densityoper(S)=\lceil \multiplicityoper(S)/\embeddingdimensionoper(S)\rceil\). We will call it the \emph{density of~\(S\)}. It gives a measure of how large is the embedding dimension of a numerical semigroup when compared to its multiplicity (or, in other words, how little is the multiplicity when compared to the embedding dimension).
  
 	The set of \emph{left elements} (respectively \emph{left primitives}) \emph{of~\(S\)} consists of the elements (respectively \emph{primitives})  \emph{of~\(S\)} that are smaller than~\(\conductoroper(S)\). 
 	The set of left elements is denoted by \(\leftsoper(S)\), or simply by~\(\leftsoper\). Occasionally the set of left primitives of \(S\) will be denoted \(\leftprimitivesoper(S)\).
 	The number of left primitives of~\(S\) is said to be the \emph{left embedding dimension} of \(S\).
 	It is denoted~\(\leftembeddingdimensionoper(S)\), or simply~\(\leftembeddingdimensionoper\).
 
 	The set of \emph{small elements} (respectively \emph{small primitives}) \emph{of~\(S\)} consists of the elements (respectively \emph{primitives}) \emph{of~\(S\)} that are smaller than or equal to~\(\conductoroper(S)\).
 
 	The terminology  \emph{right elements} (respectively \emph{right primitives}) and  \emph{big elements} (respectively \emph{big primitives}) is used with similar meanings. The notation \(\bigprimitivesoper(S)\) for the big primitives of \(S\) will occasionally be used.  
 
 	To avoid confusions, we would like to stress out that \emph{left} is used for strictly smaller than the conductor, while the use of \emph{small} is for smaller than or equal to the conductor. We use \emph{right} and \emph{big} with similar subtle differences.
 	This terminology is not original, unfortunately not always used with the same meaning by not taking into account the subtlety referred to. Our option fits the terminology in the \textsf{GAP} package \textsf{numericalsgps}~\cite{NumericalSgps1.2.0}, where \emph{small} elements are used (they even serve to define a numerical semigroup), and by Eliahou who, interested in Wilf's conjecture (stated in terms of Wilf's number (see page~\pageref{eq:wilf-number})), used the \emph{left} elements.
 	% (although he used right primitives for those primitives that are non smaller than the conductor, while we prefer to refer them as \emph{big} primitives).
 	\medskip
  
  	Pictures usually help me to understand individual (and also some families of) numerical semigroups. The numerical semigroups \(S=\langle 7,11,37,38,41\rangle\) and \(T=\langle7,18,34,37,38,40\rangle\) are depicted in Figure~\ref{fig:first-pictures}. The images were obtained automatically with the \textsf{GAP} package \textsf{intpic}~\cite{IntPic-2019}; the number of columns is the multiplicity of the semigroup. The elements of the semigroup up to \(\conductoroper+\multiplicityoper-1\) are in the coloured cells. The primitives and the conductor are further highlighted: they are in gradient coloured cells.

  	It is worth to observe that the height of the image representing a semigroup is its depth plus~1. Both semigroups in Figure~\ref{fig:first-pictures} are of depth \(5=\lceil\frac{35}{7}\rceil=\lceil\frac{34}{7}\rceil\). It is this relation between depth and height that leads us to use \(\depthoper(S)\) for the depth of \(S\).
  
  	The density of the semigroup tells how dense is the set of columns containing primitives (one at most) in the set of all columns.

	\begin{figure}[h]
% gap> ns := NumericalSemigroup(7,18,37,38,40,41);;
% gap> tkz := TikzCodeForNumericalSemigroup(ns,rec(ns_table:=true,colors:=["blue","red!70", "-red", "black!40"]));;
% gap> IP_Splash(tkz);
% gap> 
% gap> ns := NumericalSemigroup(7,18,37,38,40,41);;
% gap> tkz := TikzCodeForNumericalSemigroup(ns,rec(ns_table:=true,colors:=["blue","red!70", "-red", "black!40"]));;
% gap> IP_Splash(tkz);		
	\centering
	\begin{subfigure}[b]{0.35\textwidth}
		\includegraphics[width=\textwidth]{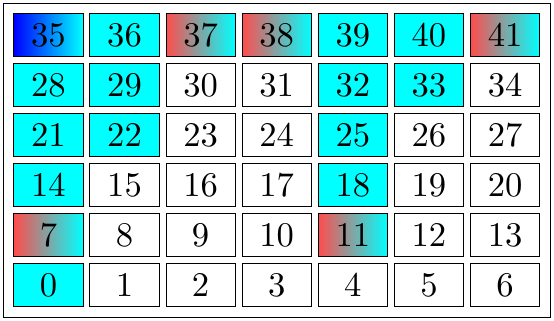}
		\caption{\(S=\langle 7,11,37,38,41\rangle\)}
	\end{subfigure}
	\quad\quad
	\begin{subfigure}[b]{0.35\textwidth}
		\includegraphics[width=\textwidth]{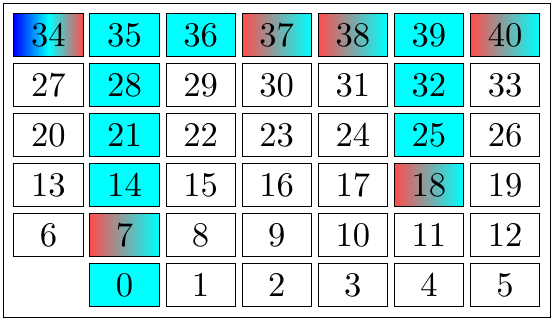}
		\caption{\(T=\langle 7,18,34,37,38,40\rangle\)}
	\end{subfigure}
	\caption{Pictorial view of numerical semigroups. \label{fig:first-pictures}}
	\end{figure} 		
 
	\medskip
 
 	The numerical semigroup \(\mathbb{N}\) has embedding dimension \(1\). All other numerical semigroups have embedding dimension at least~\(2\).
 	The multiplicity of \(\mathbb{N}\) is~\(1\) and its conductor is~\(0\). Therefore its depth is~\(0\). Its single primitive, \(1\), is a big primitive. It follows from the definitions that the multiplicity of a numerical semigroup other than \(\mathbb{N}\) does not exceed its conductor.
 	Note also that \(\mathbb{N}\) is the only numerical semigroup with no left elements, and a single small element,~\(0\).
 
 	\medskip 
 
 	A positive integer that does not belong to \(S\) is said to be a \emph{gap} of \(S\) (\emph{omitted value} in Wilf's terminology). The cardinality of the set of gaps of \(S\) is said to be its \emph{genus} and is denoted by \(\genusoper(S)\), or simply by \(\genusoper\), when \(S\) is understood. 
 
 	It is well known (see \cite[Lemma~2.14]{RosalesGarcia2009Book-Numerical}) that the genus \(\genusoper\) of a numerical semigroup is no smaller than \(\conductoroper/2\). For further reference we state it as a remark.
 	\begin{remark}\label{rem:bound_for_genus_in_terms_of_conductor}
 		With the notation above, the following holds: \(\conductoroper/2\le\genusoper\).
 	\end{remark} 

 	In the case of a numerical semigroup \(S = \langle a, b\rangle\) with embedding dimension \(2\), a very nice result of Sylvester says that \(\conductoroper(S) = (a-1)(b-1)\) and that \(\genusoper(S)=\conductoroper(S)/2\). Consequently, for semigroups of embedding dimension \(2\), one has \(\genusoper(S)/\conductoroper(S)=1/2=1-1/2\) and the following remark can be stated. (Recall that Wilf denoted by \(k\) the embedding dimension, which in the present case is~\(2\).)
 	\begin{remark}\label{rem:Silvester_implies_Wilf}
 		Wilf's equality holds for numerical semigroups with embedding dimension~\(2\). In particular, numerical semigroups with embedding dimension~\(2\) satisfy Wilf's conjecture.
 	\end{remark}

%%%%%%%%%
\subsection{A tree of numerical semigroups}\label{subsec:the_tree}
 	Consider the directed graph whose set of vertices (also called nodes) consists of all numerical semigroups, and whose edges are as follows: there is an edge from a semigroup \(S_1\) to a semigroup \(S_2\) if the latter is obtained from the former by removing one element (which is necessarily primitive). 
 
 	It is well known (see~\cite[Prop~7.1]{RosalesGarcia2009Book-Numerical}) that starting with \(\mathbb{N}\) as root and considering at each node the descendants obtained by removing big primitives (recall that these primitives are greater than the Frobenius number) one obtains a tree. It happens that this tree contains all vertices of the above defined graph, that is, all the numerical semigroups occur as vertices of this tree. A tree that contains all vertices of a graph is usually called a generating tree of the graph. 
 	This particular generating tree is frequently referred to as \emph{the} tree of numerical semigroups. 
  	It is widely considered in the literature (see, for instance,~\cite{Bras-AmorosBulygin2009SF-Towards,Elizalde2010JPAA-Improved}). 
 	We will call it \emph{the classical tree of numerical semigroups}. Sometimes we use some abbreviated form by omitting one (or both) of the terms ``classical'' or ``numerical'', for short. I am quite unsure on whether this tree always plays a more important role in the theory of numerical semigroups than any other generating tree, but the simplicity and beauty of its construction, leaves our mind not open enough to seriously consider alternatives. 
 	Terminology coming from Graph Theory used in the present paper refers to this tree or some of its subtrees, as should be clear from the context.
 
 	When a numerical semigroup \(S_2\) is obtained from a numerical semigroup \(S_1\) by removing a \emph{big} primitive, we say that \(S_2\) is a \emph{child} of \(S_1\). A numerical semigroup with no children is a \emph{leaf}. 
 	A few terms from Graph Theory, such as \emph{parent} or \emph{descendant}, will be used freely. Note that if \(S_2\ne \mathbb{N}\), then \(\Frobeniusoper(S_2)\) is a big primitive of \(S_2 \cup \{\Frobeniusoper(S_2)\}\). It follows that \(S_1\) is the parent of \(S_2\) in the numerical semigroups tree if and only if \(S_1 = S_2 \cup \{\Frobeniusoper(S_2)\}\).
 
 	It will be used the notation \(S\succ T\), or \(T\prec S\), meaning that the numerical semigroup \(T\) descends from \(S\). As usual, \(S\succeq T\) means \(S\succ T\) or \(S = T\); similarly, \(S\preceq T\) means \(S\prec T\) or \(S = T\).
 	With the notation introduced, we have that \(S\succ T\) if and only if there exists a finite sequence of numerical semigroups 
	\begin{equation}\label{eq:sequence_of_descendants}
	 	S_0=T\subset S_1\subset\cdots\subset S_r = S
	\end{equation}
  	such that \(S_{i}=S_{i-1}\cup \{\Frobeniusoper(S_{i-1})\}\), or, equivalently, \(S_{i-1}=S_{i}\setminus\{x_{i}\}\), for some \(x_i\in\bigprimitivesoper(S_i)\), with \(1\le i\le r\). 
  	Note that the sequence \(\conductoroper(S_0)> \conductoroper(S_1)>\cdots >\conductoroper(S_r)\) of the associated conductors is decreasing and that the sequence \(\genusoper(S_0)> \genusoper(S_1)>\cdots >\genusoper(S_r)\) of the associated genera is increasing. We summarize it in the following remark.
	\begin{remark}\label{rem:conductor_of_descendant}  
		Let \(S'\) and \(S\) be numerical semigroups. If \(S'\prec S\), then \(S'\subset S\), \(\conductoroper(S)< \conductoroper(S')\) and \(\genusoper(S)<\genusoper(S')\).
	\end{remark}
	Before proceeding, with the purpose of gaining familiarity with the concepts introduced, let us look at an example of a subtree rooted by a numerical semigroup.
	\begin{example}
		Let \(S=\langle 5,8,11,12,14\rangle\). Figure~\ref{fig:58111214} is a pictorial representation of the tree rooted at \(S\) (which happens to be finite). The root is on the top and the big primitives are highlighted.
		\begin{figure}[h]
%gap> SetInfoLevel(InfoViz,1);
%gap> NS := NumericalSemigroup(5,8,11,12,14);;
%gap> depth := 15;;
%gap> orientation := "down";;
%gap> tkz := tikzStringTreeRootedByNumericalSemigroup_LaTex(NS,depth,[[5,9,11,12,13]],orientation);;
%gap> Splash(tkz);
			\begin{center}
				\includegraphics[width=0.95\textwidth]{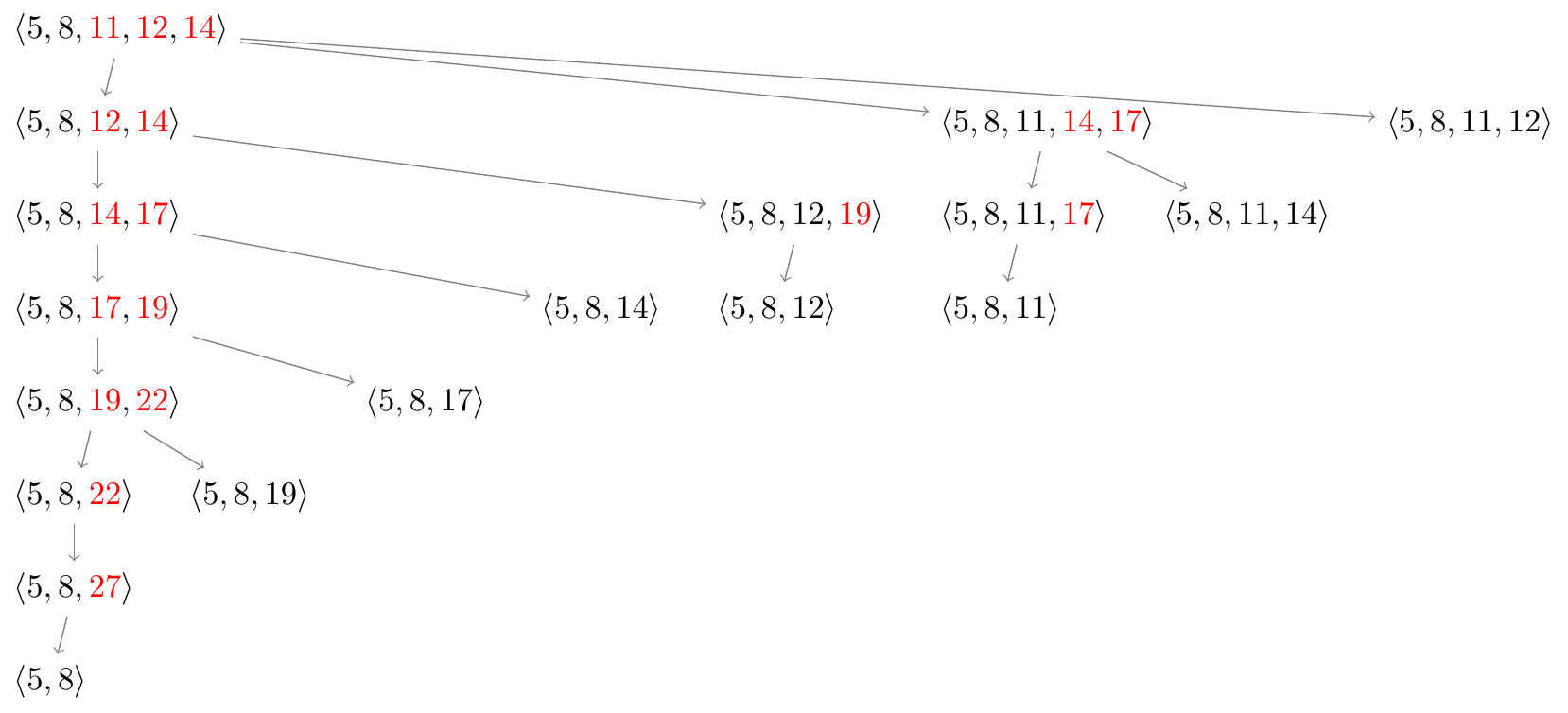}
				\caption{Pictorial representation of the tree rooted at \(S=\langle 5,8,11,12,14\rangle\).}
				\label{fig:58111214}
			\end{center}
		\end{figure} 
  		The semigroup \(\langle 5,8,19,22\rangle\) is a descendant of \(\langle 5,9,11,12,14\rangle\), which has two children, while \(\langle 5,8,17\rangle\) is a descendant of the same semigroup with no children (it is a leaf).
	\end{example} 
 	For a given positive integer \(m\), the numerical semigroup \(\mathcal{O}_m=\langle m,\ldots,2m-1\rangle\) is called the \emph{superficial semigroup of multiplicity \(m\)}, following the terminology used by Eliahou and Fromentin~\cite{EliahouFromentin2020JCTSA-Gapsets}, since the \(\mathcal{O}_m\) are exactly those numerical semigroups of depth \(1\) (\emph{ordinary} is the terminology commonly used in most of the literature presently available; \emph{half-line} was used by Rosales and García-Sánchez in~\cite{RosalesGarcia2009Book-Numerical}). Note that \(\mathcal{O}_1=\mathbb{N}\).
 	\medskip
 
 	From the above, we see that the classical tree of numerical semigroups can be constructed by successively removing \emph{big} primitives (assuming that entering infinite paths is not permitted -- exploring in a breadth first manner is appropriate; if one wants to explore in a depth first manner, a maximum level should be fixed in advance). 
 	A way to perform this construction, which is suitable for exploring in parallel, is the following.
 	First construct a forest consisting of trees rooted by superficial semigroups, then turn this forest into a tree by connecting the roots of the trees in it.
 	As the subtrees rooted by distinct superficial semigroups have no common nodes, one can explore the tree in parallel by launching one process for each of these subtrees.
 
 	A superficial semigroup of multiplicity \(m\), \(\mathcal{O}_m\),  is the root of a tree constructed as follows: the children of each node are obtained by removing big primitives greater than \(m\). 
 	Denote this tree by \(\mathbf{T}_m\). The set of vertices of \(\mathbf{T}_m\) is precisely the set of all numerical semigroups with multiplicity~\(m\). Excepting \(\mathbf{T}_1\), which has a single node, \(\mathbb{N}\), all these trees are infinite.
 
 	The classical tree of numerical semigroups is then obtained as a disjoint union of the \(\mathbf{T}_m\)	by turning it into a connected graph through the addition of the edges
  	\[\mathcal{O}_m\longrightarrow \mathcal{O}_{m+1}\]
 	Note that \(\mathcal{O}_{m+1}\) could now be seen as a result of removing from \(\mathcal{O}_m\) the big primitive \(m\) (the only big primitive that had not been removed in the construction of \(\mathbf{T}_m\)).
 
 	By using this construction, the numerical semigroups can be naturally disposed into levels: a child of a semigroup at level \(\ell\) is at level \(\ell+1\). Note that all the semigroups in a given level have the same genus. Let us fix a level for the root, \(\mathbb{N}\), (and consequently to all numerical semigroups): the root is at level \(0\). By making this choice, the level of a semigroup in the classical tree is precisely its genus. Thus, the level of the root \(\mathcal{O}_m\) of the tree \(\mathbf{T}_m\)) is \(m-1\).
 
 	Observe that the above construction involves, at each node, the determination of the primitives of each child, since these will be used to obtain the grand-children. This basically amounts to check, for each big primitive \(p\), whether \(p+\multiplicityoper\) is primitive. One finds in the literature several attempts for an efficient way of doing this (see~\cite{FromentinHivert2016MC-Exploring,Bras-AmorosFernandez-Gonzalez2018MC-Computation}; see also~\cite{Delgado2019-Probing}). 
 
 	We will vaguely refer to the construction of some subtree of the classical tree of numerical semigroups as \emph{exploring} it. When some computation has to be made for each node (testing Wilf's conjecture, for instance) we think on it as being part of the exploration.
 	This is the method used for exploring the numerical semigroups tree, or any of its subtrees, considered in this paper (so as it is in most of the literature I am aware of (e.g.~\cite{FromentinHivert2016MC-Exploring,Bras-AmorosFernandez-Gonzalez2018MC-Computation,Elizalde2010JPAA-Improved,Bras-Amoros2018inproc-Different}); exceptions are \cite{BlancoGarcia-SanchezPuerto2011IJAC-Counting,BlancoRosales2012SF-set}).
   
  	The following remark is similar to one by Bras-Amorós~\cite[Lemma~3]{Bras-Amoros2009JPAA-Bounds}, where the  interest was in big generators (effective generators in her terminology). Her proof can easily be adapted to our case. The inclusion of one proof here serves the purposes of completeness and of giving the reader the flavour of what is in cause.
  
  	\begin{lemma}\label{lemma:removing_big_primitive_edim_of_descendant}
		Let \(S\) be a numerical semigroup and suppose that \(p\) a big primitive of \(S\).
		Let \(S'=S\setminus\{p\}\). Then \(\primitivesoper(S')\) is either \(\primitivesoper(S)\setminus\{p\}\) or \((\primitivesoper(S)\setminus\{p\})\cup\{p+\multiplicityoper(S)\}\).
	\end{lemma}
	\begin{proof} 
		Clearly all primitives of \(S\) besides \(p\) are primitives of \(S'\). As \(p\) is the Frobenius number of \(S'\) and, as is well-known, none of the primitives of a numerical semigroup exceeds its Frobenius number plus its multiplicity, the only element of \(S'\) that is possibly not obtainable as a linear combination of the primitives of \(S\) that also belong to \(S'\) is \(p+\multiplicityoper(S)\).
	\end{proof}
 	\begin{corollary}\label{cor:removing_big_primitive_edim_of_descendant}
 		The embedding dimension of any semigroup in the classical tree of numerical semigroups is either the embedding dimension of its parent or one less.
	\end{corollary}

	As an immediate consequence we have that the embedding dimension decreases along a path in the semigroups tree. For further reference, we state the following remark.

	\begin{remark}\label{rem:edim_of_descendant}  
		Let \(S\) and \(T\) be numerical semigroups with \(S\succeq T\). Then \(\embeddingdimensionoper(S)\ge\embeddingdimensionoper(T)\).
	\end{remark}

	As another immediate consequence of Corollary~\ref{cor:removing_big_primitive_edim_of_descendant} we have that for a path of some length \(k\) in the semigroups tree starting in a  semigroup \(S\), the semigroup \(S'\) at the other end of the path has embedding dimension \(\embeddingdimensionoper(S')\ge\embeddingdimensionoper(S)-k\).
	The equality can occur, for any given positive integer~\(k\), as stated in the following remark. It is a consequence of Remark~\ref{rem:arbitrary_length_sequence_decreasing_edim}.
	%Example~\ref{ex:arbitrary_length_sequence_decreasing_edim}.

	\begin{remark}\label{rem:arbitrary_length_sequence_decreasing_edim}
		Let \(k\) be a positive integer. There exists a numerical semigroup \(S\) that has a descendant \(S'\) such that \(\embeddingdimensionoper(S')=\embeddingdimensionoper(S)-k\) and any semigroup in the path from \(S\) to \(S'\) (besides \(S\) itself) has embedding dimension one less than its parent.
%	\end{remark}
%	\begin{example}\label{ex:arbitrary_length_sequence_decreasing_edim}
		\smallskip
		
	    An explicit example follows.
		Set \(m=2k+1\), and let \[S=\langle\{m\}\cup\{2m+1,\ldots,2m+k\}\cup\{3m+k+1,\ldots,3m+2k\}\rangle.\] 
		Consider the sequence of length \(k\) defined recursively as follows: 
		\begin{itemize}[noitemsep]
			\item \(S_0=S\),
			\item \(S_{i+1}=S_i\setminus\{3m+k+i\}, i\in\{1,\ldots,k\}\).
		\end{itemize}
		As \(3m+k+i+m=(2m+i)+(2m+k)\), we conclude that the embedding dimension decreases along the sequence.
	
	Figure~\ref{fig:descendant_smaller_edim} illustrates the above example with \(k=4\).

	\begin{figure}[h]
		\centering
		\begin{subfigure}[b]{0.375\textwidth}
			\includegraphics[width=\textwidth]{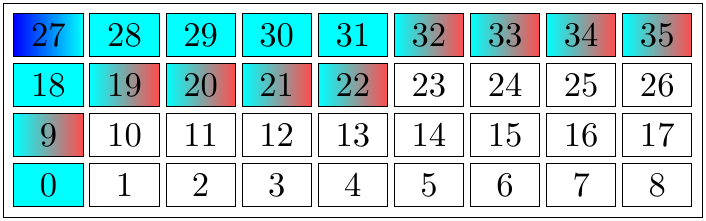}
			\caption{{\tiny\(S=\langle 9,19,\ldots,22,32,\ldots,35\rangle\)}}
		\end{subfigure}
		\begin{subfigure}[b]{0.2\textwidth}
			\includegraphics[width=\textwidth]{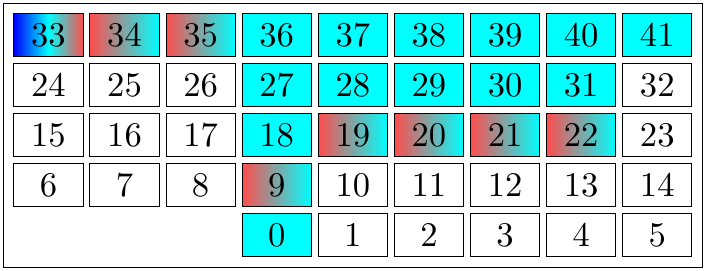}
			\includegraphics[width=\textwidth]{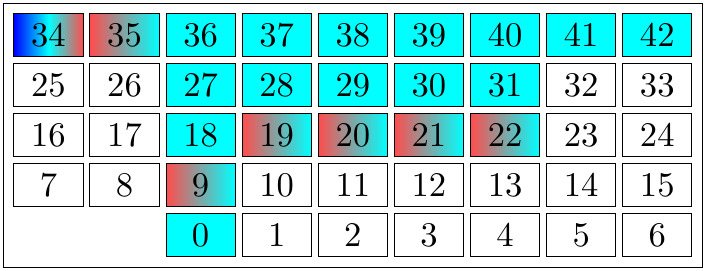}
			\includegraphics[width=\textwidth]{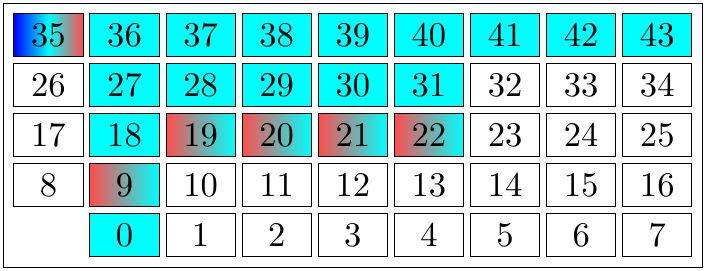}
			\caption{{\tiny\(S_2,S_2,S_4\)}}
		\end{subfigure}
		\begin{subfigure}[b]{0.375\textwidth}
			\includegraphics[width=\textwidth]{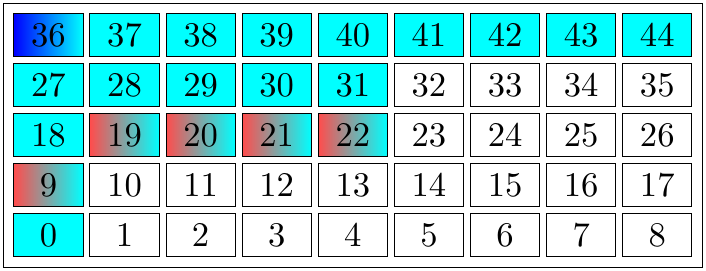}
			\caption{{\tiny\(S_4=S\setminus\{32,33,34,35\}\)}}
		\end{subfigure}
		\caption{A descendant with much smaller embedding dimension. \label{fig:descendant_smaller_edim}}
	\end{figure} 		
\end{remark}

\newpage
%%%%%%%%%%% counting %%%%%%%%%%
%\section{Counting by genus}\label{sec:counting}
%\input{counting}
%\newpage
%%%%%%%%%%% search space %%%%%%%%%%
\section{A search space}\label{sec:counting_search_space}
% counting -- search space
In this section we identify a good search space for probing conjectures on numerical semigroups or finding examples of semigroups satisfying some given property.
Such a space should fulfil some requirements, for instance being large and not consist of semigroups satisfying some (possibly restrictive) property known a priori.

A subsection is devoted to recall some known (asymptotic and experimental) results and to the presentation of some expected values. This motivates the choice of the search space, which is the subject of the second and last subsection.

%%%%%%
\subsection{Counting numerical semigroups by genus}\label{subsec:counting_by_genus}

	Related to counting by genus, some conjectures by Bras-Amorós~\cite{Bras-Amoros2008SF-Fibonacci} captured the interest of many researchers. One of her conjectures was (the presently known fact) that the sequence \(N(g)\) of the numbers of numerical semigroups of genus \(g\) shows an asymptotic Fibonacci like behaviour. This implies an exponential growth (from some unknown point on), although (to the best of my knowledge) it has not yet even been proved that the sequence \(N(g)\) grows with \(g\).
	\subsubsection{Asymptotic results due to Zhai}\label{subsubsec:asymptotic_Zhai}

  		Denote by \(N(g)\) the number of numerical semigroups of genus \(g\). The following result was conjectured by Bras-Amorós~\cite{Bras-Amoros2008SF-Fibonacci} and later proved by Zhai~\cite{Zhai2012SF-Fibonacci}.
  
  		As already referred, the survey by Kaplan~\cite{Kaplan2017AMM-Counting} gives an excellent account on counting numerical semigroups by genus, including a good outline of the strategy followed by Zhai.
  		\begin{theorem}\cite{Zhai2012SF-Fibonacci}\label{th:zhai-fibonacci}
  			\(\displaystyle \frac{N(g)}{N(g-1)} \to_g \varphi\), where \(\varphi\) is the golden number.
  		\end{theorem}
		The golden number \(\varphi=\frac{1+\sqrt{5}}{2}\) is approximately \(1.618\). In particular, \(\varphi > 1.61\). Thus Zhai's result implies that from a certain (unknown) point \(g_0\) on one has that, for \(g>g_0\), \(N(g)> 1.61 N(g-1)\). 

		Denote by \(t(g)\) the number of numerical semigroups of genus~\(g\) such that \(\conductoroper\le 3\multiplicityoper\).
		As an ingredient to be used in the proof of Theorem~\ref{th:zhai-fibonacci},
		Zhai proved the following result, which had been conjectured by Zhao~\cite{Zhao2010SF-Constructing}. 
		It says that asymptotically as the genus goes to infinity, the proportion of numerical semigroups such that \(\conductoroper\le 3\multiplicityoper\) tends to \(1\).

		\begin{theorem}\cite{Zhai2012SF-Fibonacci}\label{th:zhai-most-sgps-have-large-multiplicity}
			With the notation introduced, \(\lim_{g\to\infty} \frac{t(g)}{N(g)}=1\).
		\end{theorem}
		This result led Eliahou and Fromentin~\cite{EliahouFromentin2019ae-Gapsets} to introduce the following terminology.
		\begin{definition}\label{def:generic}
			A numerical semigroup \(S\) satisfying \(\conductoroper(S)\le 3\multiplicityoper(S)\) is called \emph{generic}. 
		\end{definition}

		Perhaps other equally important families of numerical deserve the same name (see Problem~\ref{prob:generic_embedding_dimension}). Nevertheless, this terminology, which is particularly relevant in Section~\ref{sec:motivating}, will be used along this work.
	%%%%%%%
	\subsubsection{Experimental results}
		The number \(g=70\) is the biggest one for which \(N(g)\) is presently known (to the best of my knowledge). It was obtained by Fromentin and Hivert and they made it public through a page on GitHub (see~\cite{FromentinHivert-github-page}).
		\begin{remark}\label{rem:N(70)}
			\(N(70)=\num{1607394814170158}\).
		\end{remark}
		Let \(\Theta(\Gamma)=\sum_{g=0}^{\Gamma}N(g)\) be the number of numerical semigroups with genus not greater than \(\Gamma\). The number \(\Theta(70)\) can be obtained summing up the numbers \(N(g)\) \((0\le g\le 70)\) that appear in the web-page referred above~\cite{FromentinHivert-github-page}.
		\begin{remark}\label{rem:theta(70)}	
			\(\Theta(70) = \num{4198294061955750}\).
		\end{remark}
		Let \(\Gamma_0\) be the biggest \(g\) for which \(N(g)\) is known. Note that, by Remark~\ref{rem:theta(70)}, it is not smaller than \(70\).
		The following question is related to counting.
		\begin{question}\label{challenge:genus}
			Let \(\Gamma\) be a positive integer non smaller than \(\Gamma_0+1\).
			What is the number of numerical semigroups with genus \(\Gamma\) (or up to \(\Gamma\))?
		\end{question}

		What values for \(\Gamma\) make the above problem really challenging? 
		This leads to the problem of giving an approximation to the expected time required to compute \(N(\Gamma)\) (or \(\Theta(\Gamma)\)).

	%%%%%%
	\subsubsection{(Approximated) expected numbers}
		As a naive approximation one may think that the time involved in the computations is constant for each semigroup. Better approximations, namely that take into account the magnitude of the integers that can be handled fast by, say, a \(64\) bit processor, are beyond the scope of this paper. Using the referred naive approximation, this problem on computing time is reduced to find reasonable approximations to the number of semigroups involved, and the time required by some particular computation.
		\medskip

		\emph{Assumption.}
		Assume that, for some positive integer \(\kappa\), the number \(N(\kappa)\) is known and also that \(N(g+1)\ge 1.61 \cdot N(g)\), for \(g\ge \kappa\). (This is a reasonable assumption, by the result of Zhai referred above, since \(1.61\) is smaller than the golden ratio). Then \(N(\kappa+\varepsilon)\ge 1.61^{\varepsilon}N(\kappa)\), for any nonnegative integer~\(\varepsilon\).
		As already referred the sequence of the numbers of numerical semigroups counted by genus grows exponentially (from a certain point on).
		\medskip

		The remaining part of this subsection is about approximations for \(N(100)\) and \(\theta(100)\).

		Taking \(\varepsilon=30\) one has the following approximation: \(1.61^{30} \sim \num{1602419}\).
%gap> phiP30 := Int((161/100)^30);
%1602419

		Suppose that for \(\kappa=70\) the above assumption holds. Then, using remark~\ref{rem:N(70)}, one gets the approximated expected value:

%gap> n70 := 1607394814170158;;
%gap> n100 := n70*phiP30;
%2575719990727730412202
%gap> Int(n100/10^20);
%25
		\begin{remark}\label{rem:expected_N(100)}
			\(N(100) \sim \num{1602419} \cdot  \num{1607394814170158}
			\sim 	
			\num{25} \cdot 10^{20}\).
		\end{remark}

		The same kind of reasoning can be used to obtain an approximation for the number of semigroups up to a certain genus:

		\(\Theta(\kappa+\varepsilon)=\Theta(\kappa)+(N(\kappa+1)+\cdots +N(\kappa+\varepsilon))= \Theta(\kappa)+N(\kappa)(1.61+\cdots + 1.61^{\varepsilon})\) 

		Observe that one has the following approximation \(1.61+\cdots + 1.61^{30}\sim \num{2626915}\)
%gap> Int((61*(1-1.61^30))/-100);
%977475

		Again, suppose that for \(\kappa=70\) the above assumption holds, and use Remark~\ref{rem:theta(70)}. The following expected value for \(\Theta(100)\) is obtained.

%gap> theta70 := 4198294061955750;;
%gap> phiP30 := Int((161/100)^30);
%1602419
%gap> ap := Int((100*(1-1.61^30))/-61);
%2626915
%gap> theta100 := theta70 + (n70 * ap);
%4222493746559862558320
%gap> Int(theta100/10^17);
%42224

		\begin{remark}\label{rem:expected_theta(100)}
			\(\Theta(100) \sim \num{4198294061955750} + \num{1607394814170158} \cdot \num{2626915}
	%= \num{4222493746559862558320} 
			\sim 
	%\num{42225} \cdot 10^{17}\)
			\num{42} \cdot 10^{20}\).
		\end{remark}
	
	%%%%%%%
	\subsubsection{Comments on the time required by computations}\label{subsubsec:time_for_genus_100}
		Fromentin and Hivert~\cite{FromentinHivert2016MC-Exploring} computed \(N(67)\) in \(18\) days. They provide complexity details and a full analysis of the technologies used. Let us assume that better technologies are available, faster algorithms are found or that some improvements in the implementation can be made. (See Section~\ref{sec:technical_details} for a reduction that had not been used and can expectedly reduce the time in about one half.) Being rather optimistic, suppose that one is able to compute \(N(70)\) in one day. From the above discussion one sees that computing \(N(100)\) would take \(\num{1602419}/365\sim 4390\) years. One can not be so optimistic as to expect that during this time no power failure will occur.
	%%%%%%
	\subsection{A search space}\label{subsec:a_search_space}
		This subsection is devoted to identify a reasonable search space for probing conjectures, which is the main aim of our project.
		Such a space should fulfil, at least, the following two requirements: be as big and be as free as possible. These requirements are to be seen in the following sense: consist of a large number of numerical semigroups which globally satisfy very few properties, apart from being numerical semigroups. 

		Two bad examples follow. The first of them is simple and this paper already contains enough information to completely understand it. The second is more elaborate and a complete understanding requires the knowledge of deeper results (some theoretical and one experimental); a reader who is not familiar with this topic may have to wait until Section~\ref{subsec:eliahou} to fully understand it.
		\begin{example}
			\begin{enumerate}
				\item A huge subset of the infinite set of numerical semigroups with embedding dimension \(2\) is possibly a bad choice: although consisting of a big number of semigroups, it is useless in the search for counter-examples to Wilf's conjecture, due to Remark~\ref{rem:Silvester_implies_Wilf}.
				\item A huge set of generic numerical semigroups (see Definition~\ref{def:generic}) would at first sight be a reasonable choice, since asymptotically all numerical semigroups are generic.
				But it is useless in the search for non-Eliahou semigroups, by Theorem~\ref{th:eliahou-large-mult}. Example~\ref{ex:Eliahou-Fromentin-examples} gives some examples of non-Eliahou semigroups and this suggests that there are possibly better search spaces.
			\end{enumerate}   
		\end{example} 

		Next we give an example that we consider to be good. It is the one that will be used in this paper.
		Denote by \(\mathcal{S}\) the set of numerical semigroups and, for each \(g\in \mathbb{N}\), denote by \(\mathcal{S}_g\) the set of numerical semigroups of genus \(g\) and by \(\mathcal{S}_{\le g}\) the set of numerical semigroups of genus up to \(g\).

		\begin{example}\label{ex:search_space_genus}
			Fixing some integer \(\Gamma\) between \(60\) and \(100\), say, the set \(\mathcal{S}_{\le \Gamma}\) seems to be a reasonable search space to consider. It is big, by the discussed above. Also, we are not aware of any (non artificially produced) family of numerical semigroups that has been considered by the researchers in the area all of whose elements have genus greater than \(100\), which leads us to consider \(\mathcal{S}_{\le \Gamma}\) a rather free space. Furthermore, 
			as \(\mathcal{S}=\cup_{g\in \mathbb{N}}\mathcal{S}_g\), increasing \(\Gamma\), if necessary, any non universal property has a counterexample in \(\mathcal{S}_{\le \Gamma}\).
		\end{example} 
		
		Another space could be obtained using the conductor (instead of the genus). Note that there are algorithms to determine the numerical semigroups with a given conductor (or a given Frobenius number). One implementation of such an algorithm may be found in the \textsf{numericalsgps} package (the command is \texttt{NumericalSemigroupsWithFrobeniusNumber}). Due to Remark~\ref{rem:bound_for_genus_in_terms_of_conductor} these spaces are quite similar. 
		
		The choice made in what concerns this project has been based on the genus: the search space will be  \(\mathcal{S}_{\le \Gamma}\), for some \(\Gamma\) between \(60\) and \(100\).
		
		The only reasons for this choice, between the two options just pointed out as being good, are that (to the best of my knowledge) more easily applicable results are known and more (high performance) software has been developed for computations dealing with the genus.
		
		Observe that \(\mathcal{S}_{\le \Gamma}\) consists of the semigroups in the upper part of the semigroups tree truncated at level \(\Gamma\) (which corresponds to cutting the nodes that consist of semigroups of genus bigger than \(\Gamma\)). We denote this finite tree by \(\mathbf{T}^{\Gamma}\). 

		\medskip 

		We do not know whether the search space chosen, \(\mathcal{S}_{\le \Gamma}\), is more adequate than the set of numerical semigroups with Frobenius number up to \(f\), \(\mathcal{S}_{[\le f]}\), in what concerns possible reductions or even the performance, in practice, of the algorithms involved.
		Trying to get a better understanding of this aspect is left as problem.

		\begin{problem}
			Mimic our project (for which this paper is part) by taking the set \(\mathcal{S}_{[\le f]}\) as search space instead of the set \(\mathcal{S}_{\le \Gamma}\) (for some positive integers \(f\) and \(\Gamma\)).
		\end{problem}

%%\newpage
%%%%%%%%%%% motivating problems %%%%%%%%%%
\section{Motivating problems and results}\label{sec:motivating}
%%%%%%%%%%%%
% \section{Motivating problems and results}
This section includes some motivating and very challenging problems. All of them have connections with the first problem posed by Wilf and are currently subject of active research.

The first problem is precisely Wilf's conjecture. The second one is a problem posed by Eliahou in a paper where he proved that no generic semigroup disproves Wilf's conjecture. Then we dedicate a subsection to the problem of determining for which numerical semigroups does Wilf's equality hold. Another problem is of asymptotical nature: a positive answer would show that a recent result by Eliahou showing that semigroups with embedding dimension no smaller than a third or the multiplicity are Wilf is another asymptotic proof of Wilf's conjecture.

All these problems have associated computational challenges, some of which are then explicitly posed.
These challenges will, to some extent, be addressed in the remainder of the paper.

%%%%%%%%%%%
\subsection{Wilf's conjecture}\label{subsec:wilf}
  	To a numerical semigroup~\(S\) one can associate the following number, denoted \(\wilfoper(S)\) and called the \emph{Wilf number} of~\(S\):
\begin{equation}\label{eq:wilf-number}
	 \wilfoper(S) = \lvert \primitivesoper(S) \rvert\lvert \leftsoper(S) \rvert-\conductoroper(S).
\end{equation}
 
We say that a numerical semigroup is a \emph{Wilf semigroup}  (or simply that it is \emph{Wilf}) if and only if its Wilf number is non negative. \emph{Wilf's conjecture} can be stated as follows (it is a simple exercise to verify that Wilf's conjecture is valid if and only if the first part of Wilf's Problem (a), stated in Section~\ref{subsec:Wilf_questions} has a positive answer):
\begin{conjecture}[Wilf, 1978]\label{conj:wilf}
 	Every numerical semigroup is a Wilf semigroup.
\end{conjecture}
  
As mentioned, the conjecture is a quite popular topic among the researchers in the area, and considerable work has been done.
For the purpose of this paper, I just refer three particular cases, which happen to be of high relevance. References to a number of other particular cases can be found in~\cite{Delgado2019ae-Conjecture}.
 
Fromentin and Hivert verified computationally the conjecture for the \(\num{33474094027610}\) semigroups of genus up to \(60\).
 
\begin{proposition}[\cite{FromentinHivert2016MC-Exploring}]\label{prop:genus-60-are-Wilf}
 	All numerical semigroups of genus up to \(60\) are Wilf.
\end{proposition}

% Eliahou %
Kaplan~\cite{Kaplan2012JPAA-Counting} proved that all numerical semigroups with \(\conductoroper(S)\le 2\multiplicityoper(S)\) are Wilf. Using as main ingredient a theorem of Macaulay on the growth of Hilbert functions of standard  graded algebras, Eliahou extended that result to the case \(\conductoroper(S)\le 3\multiplicityoper(S)\):
 
\begin{theorem}[{\cite[Cor.~6.5]{Eliahou2018JEMS-Wilfs}}]\label{th:large-mult-are-wilf}
	Generic semigroups are Wilf.
\end{theorem}
One can also suggestively refer the above result by saying that semigroups of small depth are Wilf, where ``small depth'' should be understood as ``depth \(\le 3\)''.
 
Another deep result of Eliahou is stated next. It involves the embedding dimension, another of the invariants present in most of the literature on numerical semigroups. In his proof, Eliahou uses the elements of a numerical  semigroup to associate a graph to it. Then, graph theory intervenes in a non trivial way, namely by using the notion of vertex-maximal matching. The result obtained improves one of Sammartano~\cite{Sammartano2012SF-Numerical} which states that a numerical semigroup such that \(\left| \primitivesoper \right| \ge m / 2\) is Wilf. 
\begin{theorem}[\cite{Eliahou2019ae-graph}]\label{th:large-ed-are-wilf}
 	Let \(S\) be a numerical semigroup such that \(\embeddingdimensionoper(S) \ge \multiplicityoper(S) / 3\). Then \(S\) is Wilf.
\end{theorem}
A suggestive way to refer to this result is by saying that semigroups of \emph{large} embedding dimension (with respect to the multiplicity) are Wilf. Recall that the embedding dimension of a numerical semigroup does not exceed its multiplicity, which motivates the use of the term ``large''.
 
The above results, especially Theorem~\ref{th:large-ed-are-wilf}, play a crucial role in our project, as will become apparent from the remaining of the paper. 
 \medskip
 
Examples~\ref{example:large_depth} and~\ref{example:m80} below are used to observe that the families of numerical semigroups guaranteed to be Wilf by Theorem~\ref{th:large-mult-are-wilf} and by Theorem~\ref{th:large-ed-are-wilf} are considerably different.
The first of these examples gives explicitly an infinite family of large embedding dimension numerical semigroups  whose depth is large. In particular, all members of this family
are Wilf by Theorem~\ref{th:large-ed-are-wilf}, but the same conclusion can not be achieved by using Theorem~\ref{th:large-mult-are-wilf}.
\begin{example}\label{example:large_depth}
  	Let \(m,n\) be positive integers, with \(m\ge 2\). Then the conductor of the semigroup \(\langle m, nm+1,\ldots, nm + m-1 \rangle\) is \(nm\). Thus, the infinite set 
	\(\{\langle m, nm+1,\ldots, nm + m-1 \rangle\colon n \ge 4\}\) consists of semigroups whose depth is greater than \(3\).
\end{example}
 
An infinite family of generic and small embedding dimension is given in the following example, borrowed from~\cite{Delgado2019ae-Conjecture} where details may be found. 
 
\begin{example}\label{example:m80}
 	Let \(m\) be a positive integer and let \(Y_m=\{m, m+1,m+2,m+3\}\cup \{7k+m\mid 0 < k\le \lceil \frac{m}{7}\rceil\}\). Consider the semigroup \(S_m=\langle Y_m\rangle\). Figure~\ref{fig:M-not-ED-m80} give a pictorial representation.	
 	
 	\begin{figure}[h]
 		\begin{center}
 			\includegraphics[width=\textwidth]{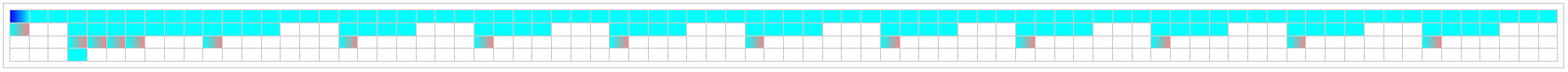}
 		\end{center}
 		\caption{Shape of the numerical semigroup \(S_{80}\). \label{fig:M-not-ED-m80}}
 	\end{figure} 
 	
 	It is easy to check that \(\primitivesoper(S_m) = \{m, m+1,m+2,m+3\}\cup \{7k+m\mid 0 < k\le \lceil \frac{m}{7}\rceil-1\}\). Thus the embedding dimension of \(S_m\) is \(4+ \lceil\frac{m}{7}\rceil-1= 3+\lceil \frac{m}{7}\rceil\). Note that \(3+\frac{m}{7} \le 3+\lceil \frac{m}{7}\rceil\), and that \(3+\frac{m}{7}  <\frac{m}{3}\) if and only if \(7m-3m > 3\times 7\times 3\), that is \(4m>63\). Thus one concludes that, for \(m>15\), \(S_m\) is not of large embedding dimension.	
\end{example}

 The above example shows, in particular, that there are infinitely many semigroups that Theorem~\ref{th:large-mult-are-wilf} guarantees to be Wilf while Theorem~\ref{th:large-ed-are-wilf} does not provide such guarantee.

%%%%%%%
\subsection{A problem by Eliahou}\label{subsec:eliahou}
  	Eliahou semigroups are defined in a way that is similar to the one we used to define Wilf semigroups. 
    Recall that the integer
  	\(\depthoper=\lceil \conductoroper/\multiplicityoper\rceil\) is called the depth 
  	of the numerical semigroup. Denote by \(\rhooper=\depthoper\cdot \multiplicityoper-\conductoroper\) the (negative) remainder of the division of the conductor by the multiplicity.  
  	The set of decomposable elements in the interval of length \(\multiplicityoper\) starting in \(\conductoroper\)  (which Eliahou suggested to name the \emph{threshold interval}) is \(\Ddepthoper=\{n\in \mathbb{N}\mid \conductoroper\le n < \conductoroper + \multiplicityoper\}\setminus \primitivesoper\).
 
 	To a numerical semigroup~\(S\) one associates the following number denoted \(\eliahouoper(S)\) and called the  \emph{Eliahou number of~\(S\)}:
 	\begin{equation}\label{eq:eliahou-number}
 		\eliahouoper(S) = \lvert \primitivesoper\cap \leftsoper\rvert\lvert \leftsoper \rvert - {\depthoper} \lvert \Ddepthoper\rvert +\rho.
 	\end{equation}
 
	A numerical semigroup is said to be an  \emph{Eliahou semigroup} if its Eliahou number is non negative. 
 
	\begin{proposition}[{\cite[Prop.~7.6]{Eliahou2018JEMS-Wilfs}}]\label{th:eliahou-large-mult}
	 	Generic semigroups are Eliahou.
	\end{proposition}
 
	An important feature of Eliahou semigroups is given by the following proposition.
 
\begin{proposition}[{\cite[Prop.~3.11]{Eliahou2018JEMS-Wilfs},\cite[Cor.~2.3]{EliahouFromentin2019SF-misses}}]\label{prop:eliahou-implies-wilf}
 	Eliahou semigroups are Wilf.
\end{proposition}

Note that Theorem~\ref{th:large-mult-are-wilf} is an immediate consequence of the two propositions above. This reflects in fact the strategy followed by Eliahou to obtain that great result.
\medskip

The converse of Proposition~\ref{prop:eliahou-implies-wilf} does not hold. In fact, up to genus \(60\) there are \(5\) non-Eliahou semigroups (obtained using exhaustive search by Fromentin, as referred in~\cite{Eliahou2018JEMS-Wilfs}). 
\begin{example}[Fromentin]\label{ex:Eliahou-Fromentin-examples} 
	The following are the only numerical semigroups up to genus \(60\) with negative Eliahou number (which happens to be \(-1\) in all cases):
	\(\langle14,22,23\rangle_{56}\), \(\langle16,25,26\rangle_{64}\), \(\langle17,26,28\rangle_{68}\), \(\langle17,27,28\rangle_{68}\) and \(\langle18,28,29\rangle_{72}\).
\end{example}
 
It is currently known that there are infinitely many non-Eliahou semigroups (Delgado~\cite{Delgado2018MZ-question}, Eliahou and Fromentin~\cite{EliahouFromentin2019SF-misses}). One even knows that for any given integer there are infinitely many semigroups whose Eliahou number is that integer (see~\cite{Delgado2018MZ-question}). All of those known non-Eliahou semigroups have been proven to be Wilf.
 
A most probably very difficult problem is open.
\begin{problem}[{\cite[Sec. 3.5]{Eliahou2018JEMS-Wilfs}}]\label{prob:characteriza-Eliahou}
	Characterize Eliahou semigroups.
\end{problem}
	From the above discussion, it is reasonable to expect that the characterization Problem~\ref{prob:characteriza-Eliahou} asks for would give a huge contribution to the solution of Wilf's conjecture.

%%%%%%% 
\subsection{Semigroups with Wilf number 0}\label{subsec:0-wilf}
	Observe that the Wilf number of a numerical semigroup is \(0\) if and only if Wilf's equality holds for that semigroup. A numerical semigroup with Wilf number \(0\) is said to be \emph{\(0\)-Wilf}.    

	Let \(k\) and \(m\) be positive integers. A semigroup of the form \(\langle m, km+1,\ldots (k+1)m-1\rangle\) is said to be \emph{quasi-superficial}.
	It is easy to observe that quasi-superficial numerical semigroups are \(0\)-Wilf (see the final comments made by Fröberg, Gottlieb and Haeggkvist in~\cite{FroebergGottliebHaeggkvist1987SF-numerical}). As already observed (Remark~\ref{rem:Silvester_implies_Wilf}), numerical semigroups of embedding dimension~\(2\) are \(0\)-Wilf.

	Using the terminology just introduced, a question posed by Moscariello and Sammartano~\cite[Question~8]{MoscarielloSammartano2015MZ-conjecture} can be stated as follows:

	\begin{problem}\label{prob:wilf-equality}
		Let \(S\) be a \(0\)-Wilf numerical semigroup. Is it true that either \(\embeddingdimensionoper(S)=2\) or \(S\) is quasi-superficial?
	\end{problem}
	We will refer to this question as the \emph{\(0\)-Wilf problem.}
	Moscariello and Sammartano observed that the \(0\)-Wilf problem has no negative answer among the numerical semigroups up to genus \(35\). (They used to this effect the package \textsf{numericalsgps}, which may be seen as general purpose software. Bigger numbers will be provided in the experimental parts of this project.)

	Kaplan~\cite[Prop.~26]{Kaplan2012JPAA-Counting} has shown that Problem~\ref{prob:wilf-equality} has a positive answer in the case of numerical semigroups whose multiplicity is at least half of the conductor. The same holds for numerical semigroups of depth \(3\) (see a remark by Sammartano in~\cite[Rem.~6.6]{Eliahou2018JEMS-Wilfs}). In other words, there is is no negative answer for the \(0\)-Wilf problem among generic semigroups.

	\begin{proposition}[{\cite[Rem.~6.6]{Eliahou2018JEMS-Wilfs}}]\label{prop:generic_are_0_Wilf}
		Let \(S\) be a generic \(0\)-Wilf numerical semigroup. Then either \(\embeddingdimensionoper(S) = 2\) or \(S\) is quasi-superficial. 
	\end{proposition}

%%%%%%%
\subsection{Consequences of an asymptotic result and a related problem}\label{subsec:asymptotic_a_result_and_a_problem}
	We recall Theorem~\ref{th:zhai-most-sgps-have-large-multiplicity}, an asymptotic result due to Zhai. The following notation was used: given a positive integer \(g\), \(t(g)\) denotes the number of generic semigroups of genus \(g\) and \(N(g)\) denotes the total number of semigroups of genus \(g\). Zhai's result says:
	\[\lim_{g\to\infty}\frac{t(g)}{N(g)}=1.\]
	In this sense, asymptotically, all numerical semigroups are generic.
	Theorems~\ref{th:large-mult-are-wilf}, Propositions~\ref{th:eliahou-large-mult} and~\ref{prop:generic_are_0_Wilf} state that all generic semigroups are Wilf, are Eliahou and do not provide any negative answer to Problem~\ref{prob:wilf-equality}. Thus, the mentioned results give, in particular, an asymptotic proof of Wilf's conjecture and an asymptotic positive answer to the \(0\)-Wilf problem. This should be used to better appreciate those results.

	\medskip

	A natural question is whether there is an asymptotic result similar to Theorem~\ref{th:zhai-most-sgps-have-large-multiplicity}, but considering the number of semigroups of large embedding dimension rather than the generic ones.
	More precisely:

	\begin{problem}\label{prob:generic_embedding_dimension}
		Denote by \(E(g)\) the number of numerical semigroups of genus \(g\) satisfying \(\embeddingdimensionoper \ge \multiplicityoper/3\). Is it true that \(\lim_{g\to\infty} \frac{E(g)}{N(g)}=1\)?
	\end{problem}

	A positive answer to Problem~\ref{prob:generic_embedding_dimension} would lead to a very interesting consequence: one would be able to conclude that Theorem~\ref{th:large-ed-are-wilf} provides another proof of the asymptotic validity of Wilf's conjecture. Let us stress out that the proofs obtained by Eliahou of Theorems~\ref{th:large-mult-are-wilf} and~\ref{th:large-ed-are-wilf} involve radically different arguments, thus one would have two independent proofs of the asymptotic validity of Wilf's conjecture. See also Examples~\ref{example:large_depth} and~\ref{example:m80} and the comments nearby, which point out that the results are rather different.

	Experiments suggest that the answer to Problem~\ref{prob:generic_embedding_dimension} is in fact positive. Table~\ref{table:by_genus_low_edim} collects some numbers: for each \(g\in\mathbb{N}\) non greater than \(45\) it indicates the number \(N(g)\) of numerical semigroups and the number \(led(g)\) of those among them with little embedding dimension (that is, satisfying \(\embeddingdimensionoper<\multiplicityoper/3\)). Summing up the columns \(led(g)\) and \(N(g)\), one gets that there are \(\num{202491}\) numerical semigroups of little embedding dimension among the \(\num{23022228615}\) numerical semigroups of genus up to \(45\). The quotient \(\frac{\num{202491}}{\num{23022228615}}\), approximately \(8.8\cdot 10^{-6}\), is rather small.
	(For longer tables see~\cite{Delgado2019-Probing}.)

	\begin{table}
  		\begin{tabular}{r|r|r||r|r|r||r|r|r||r|r|r}
		\hline
		\(g\) &	\(N(g)\) & \(led(g)\)& \(g\) & \(N(g)\)& \(led(g)\) &\(g\) & \(N(g)\)& \(led(g)\)&\(g\) & \(N(g)\)& \(led(g)\)\\
		\hline
		\num{0} & \num{1} & \num{1} & \num{12} & \num{592} & \num{0} & \num{24} & \num{282828} & \num{4} & \num{36} & \num{109032500} & \num{1298}\\
		\num{1} & \num{1} & \num{0} & \num{13} & \num{1001} & \num{0} & \num{25} & \num{467224} & \num{11} & \num{37} & \num{178158289} & \num{2096}\\
		\num{2} & \num{2} & \num{0} & \num{14} & \num{1693} & \num{0} & \num{26} & \num{770832} & \num{11} & \num{38} & \num{290939807} & \num{3267}\\
		\num{3} & \num{4} & \num{0} & \num{15} & \num{2857} & \num{0} & \num{27} & \num{1270267} & \num{17} & \num{39} & \num{474851445} & \num{5058}\\
		\num{4} & \num{7} & \num{0} & \num{16} & \num{4806} & \num{0} & \num{28} & \num{2091030} & \num{26} & \num{40} & \num{774614284} & \num{8073}\\
		\num{5} & \num{12} & \num{0} & \num{17} & \num{8045} & \num{0} & \num{29} & \num{3437839} & \num{43} & \num{41} & \num{1262992840} & \num{12208}\\
		\num{6} & \num{23} & \num{0} & \num{18} & \num{13467} & \num{0} & \num{30} & \num{5646773} & \num{117} & \num{42} & \num{2058356522} & \num{19308}\\
		\num{7} & \num{39} & \num{0} & \num{19} & \num{22464} & \num{0} & \num{31} & \num{9266788} & \num{108} & \num{43} & \num{3353191846} & \num{29901}\\
		\num{8} & \num{67} & \num{0} & \num{20} & \num{37396} & \num{0} & \num{32} & \num{15195070} & \num{222} & \num{44} & \num{5460401576} & \num{46835}\\
		\num{9} & \num{118} & \num{0} & \num{21} & \num{62194} & \num{3} & \num{33} & \num{24896206} & \num{323} & \num{45} & \num{8888486816} & \num{72076}\\
		\num{10} & \num{204} & \num{0} & \num{22} & \num{103246} & \num{1} & \num{34} & \num{40761087} & \num{576} & \num{} & \num{} & \num{}\\
		\num{11} & \num{343} & \num{0} & \num{23} & \num{170963} & \num{0} & \num{35} & \num{66687201} & \num{908} & \num{} & \num{} & \num{}\\
			\hline
	\end{tabular}
		\caption{Number of semigroups with little embedding dimension, by genus.  \label{table:by_genus_low_edim}}
	\end{table}

	The search for an asymptotic result concerning embedding dimension as the one referred in Problem~\ref{prob:generic_embedding_dimension} is clearly a theme for possible future work.

%%%%%%%
\subsection{Challenges}\label{subsec:challenges}
	Section~\ref{sec:motivating} gives rise to several challenges that are worth to consider from computational and experimental points of view. Before developing some theory, which will happen in later sections, we will list the ones we want to consider.
	\medskip
	
	In what concerns the counting of all numerical semigroups up to a given genus, a challenge was already presented in Question~\ref{challenge:genus}.
	
	Counting in other classes of numerical semigroups is also of high interest. For instance, Eliahou and Fromentin~\cite{EliahouFromentin2019} counted generic numerical semigroups up to genus~\(65\). (For the numbers up to genus \(60\), see~\cite[Fig.~4]{EliahouFromentin2020JCTSA-Gapsets}).
	
	\begin{question}\label{challenge:count_generic}	
		Count the (non-)generic semigroups up to a genus as big as possible.
	\end{question}
	Related with Problem~\ref{prob:generic_embedding_dimension}, one may ask for the number of numerical semigroups of some given genus with large embedding dimension, that is, \(d\ge \multiplicityoper/3\). Up to genus \(45\) Table~\ref{table:by_genus_low_edim} provides the answer (since, for each genus \(g\), the desired number is \(N(g)-led(g)\)). %(\commentblue{large+little=total})	
	\begin{question}\label{challenge:count_large_edim}	
		Count the numerical semigroups with little (or with large) embedding dimension up to a genus as large as possible.
	\end{question}

	Recall that with genus up to \(60\) all numerical semigroups are Wilf. Next we state a natural challenge as a question.
	\begin{question}\label{challenge:wilf}
		Let \(\Gamma\) be a positive integer non smaller than \(61\).
		Are there non Wilf semigroups of genus up to \(\Gamma\)?
	\end{question}

	From Example~\ref{ex:Eliahou-Fromentin-examples}, we know that up to genus \(60\) there are \(5\) non-Eliahou semigroups. A natural challenge is next stated as a question (whose answer naturally includes the \(5\) examples mentioned).
	\begin{question}\label{challenge:eliahou}
		Let \(\Gamma\) be a positive integer non smaller than \(61\).
		Which are the non-Eliahou semigroups of genus up to \(\Gamma\)?
	\end{question}
	
	Another natural challenge, this related to Problem~\ref{prob:wilf-equality}, is the following: 
	\begin{question}\label{challenge:Wilf_number_0}
		Let \(\Gamma\) be a positive integer non smaller than \(35\). Is there any non-superficial numerical semigroup with embedding dimension greater than \(2\), Wilf number \(0\) and genus up to \(\Gamma\)?
	\end{question}

%%%%%%
\subsection{The need of reducing the search space}
	In Section~\ref{subsubsec:time_for_genus_100} we referred that, as there are many numerical semigroups with genus up to \(100\), counting them using the available means would take no less than \(\num{4000}\) years. 
	
	And what happens if in addition to counting one wants to test some property (for instance ``is Wilf'') for all those semigroups? It may happen that, using an appropriate encoding (see Section~\ref{sec:encoding_with_redundancy}) testing a property is not very time consuming.
	But even assuming that it costs no extra time, obtaining computational results up to genus \(100\) without reducing the search space is out of reach.	
	
	This paper aims to be a contribution to, through the integration of theoretical results, reduce the search space (depending on the property under consideration).
	
	When the property in cause is ``is Wilf'', that is, when considering Question~\ref{challenge:wilf}, \(\Gamma=100\) is reachable. Due to Theorem~\ref{th:large-ed-are-wilf}, only searching among numerical semigroups of low embedding dimension is involved.
	For records, or updated values, see~\cite{DelgadoFromentin2019}.
	
	All other questions stated in the present section involve searching among non-generic semigroups, which seems to require more costly computational tasks. This leads us to be not so optimistic as in the case of Wilf's conjecture for which we can use the cheap operation of looking to the number of primitives. Nevertheless, a lot can be done.

%%\newpage
%%%%%%%%%%% properties %%%%%%%%%%
%\section{Some properties}\label{sec:properties}
%\input{properties}
%\newpage
%%%%%%%%%%% properties %%%%%%%%%%
\section{Left elements}\label{sec:lefts}
% Left elements
This section is elementary and could appear well before in this paper. Our option has been to put it near the place where the results contained in it will start to be used.
It collects some observations on left elements or left primitives and the semigroup they generate.

As usual, \(S\) is an arbitrary numerical semigroup.
Recall that the set of left elements of \(S\) consists of the elements to the left of the conductor, that is, \(\leftsoper(S)=S\setminus\{\conductoroper(S),\rightarrow\}\).
Recall also that the number \(\left|\leftprimitivesoper(S)\right|\) of left primitives is called the left embedding dimension of \(S\) and is denoted \(\leftembeddingdimensionoper(S)\).

\begin{lemma}\label{lemma:lefts_elements_primitives}
	The left elements and the left primitives of a numerical semigroup \(S\) generate the same semigroup. In symbols: \(\langle\leftsoper(S)\rangle=\langle\leftprimitivesoper(S)\rangle\).
\end{lemma}
\begin{proof}
	Clearly \(\leftprimitivesoper(S)\subseteq\leftsoper(S)\subseteq\langle\leftprimitivesoper(S)\rangle\). Taking, for each of the sets in the chain, the semigroup it generates, we get  
	\begin{equation}\label{eq:lefts_elements_primitives}
		\langle\leftprimitivesoper(S)\rangle\subseteq\langle\leftsoper(S)\rangle\subseteq\langle\langle\leftprimitivesoper(S)\rangle\rangle.
	\end{equation} 
	As \(\langle\langle\leftprimitivesoper(S)\rangle\rangle=\langle\leftprimitivesoper(S)\rangle\), both ends of the chain (\ref{eq:lefts_elements_primitives}) are equal, and the result follows.
\end{proof}

It follows from the definitions that a numerical semigroup consists of its left elements and all the integers non smaller than the conductor. It is also clear that \(S = \leftsoper(S)\cup \mathcal{O}_{\conductoroper(S)}=\langle\leftsoper(S) \rangle_{\conductoroper(S)}\). Using Lemma~\ref{lemma:lefts_elements_primitives}, we can write the following remark:

\begin{remark}\label{rem:lefts_primitives_truncated} 
	\(S = \langle\leftprimitivesoper(S) \rangle_{\conductoroper(S)}\).
\end{remark}
Another immediate consequence of Lemma~\ref{lemma:lefts_elements_primitives} is stated next. Observe that \(\langle\leftsoper(S)\rangle\) is a numerical semigroup if and only if \(\gcd(\leftsoper(S))\ne 1\). 
\begin{remark}\label{rem:edim_equals_edim_lefts} 
	If \(\gcd(\leftsoper(S))=1\), then  \(\leftembeddingdimensionoper(S) = \embeddingdimensionoper(\langle\leftsoper(S)\rangle)\).
\end{remark}
Recall that a leaf is a numerical semigroup with no big primitives. Thus, a leaf is generated by the left primitives. By Lemma~\ref{lemma:lefts_elements_primitives}, we can write the following:
\begin{remark}\label{rem:leaf} 
	\(S\) is a leaf if and only if \(S = \langle\leftsoper(S)\rangle\).
\end{remark}
If \(S\) is not a leaf, then the child of \(S\) obtained by removing the smallest big primitive clearly has the same left primitives as \(S\). The process can be repeated for a chain of descendants, as the following remark states.
\begin{remark}\label{rem:removing_smallest_big_primitive}
	Suppose that \(S\) has some big primitive and let 
	\(S=S_0\succ S_1\succ \cdots \succ S_r\) be a chain of descendants of \(S\)
	such that \(S_i=S_{i-1}\setminus\{\min(\bigprimitivesoper(S_{i-1}))\}\) (\(1\le i\le r\)).
	Then \(\leftprimitivesoper(S_i)=\leftprimitivesoper(S)\), for all \(1\le i\le r\).
\end{remark}
 
\begin{lemma}\label{lemma:descendants_contain_left_elements}
	Let \(S\) and \(S'\) be numerical semigroups.
	If \(S'\preceq S\), then \(\leftsoper(S)\subseteq S'\).
\end{lemma}
\begin{proof}
	The result is obvious if \(S'=S\). Suppose that \(S'\prec S\).
	There exists a finite sequence \(S=S_0\succ S_1\succ \cdots \succ S_r = S'\) of descendants of \(S\) that ends in \(S'\) such that each intermediate semigroup is a child of the preceding one in the sequence. That is, \(S_i\) (\(1\le i\le r\)) is obtained from \(S_{i-1}\) by removing a primitive that is no smaller than \(\conductoroper(S_{i-1})\). Note that \(\conductoroper(S_{i-1})\ge \conductoroper(S)\), by Remark~\ref{rem:conductor_of_descendant}. Thus the set of left elements of \(S\) is not touched along the sequence and therefore is contained in all semigroups of the sequence. 
\end{proof}
As a consequence of the preceding lemma and Remark~\ref{rem:primitive_in_subsemigroup} we get the following:
\begin{corollary}\label{cor:lefts_sequence_increases}
	If \(S'\preceq S\), then \(\leftprimitivesoper(S)\subseteq\leftprimitivesoper(S')\). In particular, \(\leftembeddingdimensionoper(S')\ge \leftembeddingdimensionoper(S)\).
\end{corollary}
	The next result characterizes the descendants of a numerical semigroup in terms of elements to the left of \(\conductoroper(S)\): numerical semigroups descending from \(S\) are those that coincide with \(S\) in the interval \(\left[0,\Frobeniusoper(S)\right]\).
	
\begin{proposition}\label{prop:descendants_correspondance_with_lefts}
	Let \(S\) and \(S'\) be numerical semigroups. The following conditions are equivalent:
	\begin{enumerate}
		\item \(S'\preceq S\);
		\item \(\leftsoper(S)=S\cap\left[0,\Frobeniusoper(S)\right]=S'\cap\left[0,\Frobeniusoper(S)\right]\).
	\end{enumerate}
\end{proposition}
\begin{proof}
	\((1)\Rightarrow(2)\):
	By definition we have that \(\leftsoper(S)=S\cap\left[0,\Frobeniusoper(S)\right]\), thus we have to prove that 
	\(S\cap\left[0,\Frobeniusoper(S)\right]=S'\cap\left[0,\Frobeniusoper(S)\right]\).
	
	Recall that \(S'\subseteq S\) (Remark~\ref{rem:conductor_of_descendant})
	Thus, \(S'\cap\left[0,\Frobeniusoper(S)\right]\subseteq S\cap\left[0,\Frobeniusoper(S)\right]=\leftsoper(S)\). The reverse inclusion follows from Lemma~\ref{lemma:descendants_contain_left_elements}, which states that \(\leftsoper(S)\subseteq S'\): in fact, intersecting both sides with \(\left[0,\Frobeniusoper(S)\right]\), we get
	 \(\leftsoper(S)\cap\left[0,\Frobeniusoper(S)\right]\subseteq S'\cap\left[0,\Frobeniusoper(S)\right]\).
	
	\medskip
	
	\((2)\Rightarrow(1)\):
	We claim that if \(S'\ne S\) and \(S\cap\left[0,\Frobeniusoper(S)\right]=S'\cap\left[0,\Frobeniusoper(S)\right]\), then \(\Frobeniusoper(S)<\Frobeniusoper(S')\).
	
	Proof of the claim. We will prove that neither \(\Frobeniusoper(S)=\Frobeniusoper(S')\) nor \(\Frobeniusoper(S)>\Frobeniusoper(S')\) can occur.
	
	If \(\Frobeniusoper(S)=\Frobeniusoper(S')\), then \(\leftsoper(S)=S\cap\left[0,\Frobeniusoper(S)\right]=S'\cap\left[0,\Frobeniusoper(S)\right]=S'\cap\left[0,\Frobeniusoper(S')\right]=\leftsoper(S')\). Thus \(S\) and \(S'\) have the same left elements and the same conductor. This implies \(S=S'\), which is a contradiction.
	
	If \(\Frobeniusoper(S)>\Frobeniusoper(S')\), then \(\Frobeniusoper(S)\ge \conductoroper(S')\) and therefore \(\Frobeniusoper(S) \in S'\cap\left[0,\Frobeniusoper(S)\right]\). But \(\Frobeniusoper(S) \not\in S\cap\left[0,\Frobeniusoper(S)\right]\), which is a contradiction. The proof of the claim is complete.
	\medskip
	
	We will complete the proof by constructing a chain of ancestors of \(S'\) that ends in \(S\). 
	Assume that \(S\ne S'\) (if \(S=S'\), there is nothing to prove).
	
	Note that \(S\setminus S' \subseteq \left[\conductoroper(S),\Frobeniusoper(S')\right]\), since up to \(\Frobeniusoper(S)\) the semigroups coincide by hypothesis and both semigroups contain all integers greater that \(\Frobeniusoper(S')\), by the claim just proved. Thus \(S\setminus S' = \left[\conductoroper(S),\Frobeniusoper(S')\right]\setminus S'\).
	Thus, \(\Frobeniusoper(S')=\max(S\setminus S')\), and \(S''=S'\cup \{\Frobeniusoper(S')\}\) is the parent of \(S'\). If \(S''=S\), then the process ends; otherwise it can be continued with \(S''\) replacing \(S'\). This is a finite process since \(S\setminus S'\) is finite.
\end{proof}
	
The previous proposition has various interesting consequences.
	
\begin{corollary}\label{cor:lefts_is_descendant}
	If \(S\) is a numerical semigroup such that \(\gcd(\leftsoper(S))=1\), then \(\langle\leftsoper(S)\rangle\preceq S\).
\end{corollary}
\begin{proof}
	Take \(S'=\langle\leftsoper(S)\rangle\) in Proposition~\ref{prop:descendants_correspondance_with_lefts}.
\end{proof}
	
The following corollary gives a way to construct descendants with bigger left embedding dimension (when it is possible).

\begin{corollary}\label{cor:a_descendant}
	Let \(S\) be a non-superficial numerical semigroup and let \(p\) be a prime integer non smaller than \(\conductoroper(S)\). 
	Then \(T=\langle \leftsoper(S)\cup \{p\}\rangle\preceq S\).% is either \(S\) or a descendant of \(S\).
\end{corollary}
\begin{proof}
	As \(S\) is non-superficial, \(\leftsoper(S)\) contains some positive integer.
	Now observe that a positive integer is coprime to any prime greater than it. Thus \(\leftsoper(S)\cup \{p\}\) is a set of globally coprime positive integers and therefore \(T\) is a numerical semigroup. If \(T=S\), nothing remains to be proved. Otherwise, 
	the result follows from Proposition~\ref{prop:descendants_correspondance_with_lefts}.
\end{proof}

A numerical semigroup has finitely many descendants if and only if its left elements are globally coprime, as stated by Bras-Amorós and Bulygin. 
The following arguments, included for the sake of completeness, constitute a proof of this fact, which is then stated as a proposition.

Let \(S\) be a numerical semigroup and let
 \(\mathcal{H}=\left\{\conductoroper(S),\rightarrow\right\}\setminus\langle\leftsoper(S)\rangle.\)
It follows from Proposition~\ref{prop:descendants_correspondance_with_lefts} that the set
\[\Delta(S)=\{\langle\leftsoper(S)\rangle_h\mid h\in \mathcal{H}\}\]
consists of descendants of \(S\), since \(\leftsoper(S)=\langle\leftsoper(S)\rangle_h\cap\left[0,\Frobeniusoper(S)\right]\), for any integer \(h\ge\conductoroper(S)\).

Note that \(\mathcal{H}\) is finite if and only if \(\langle\leftsoper(S)\rangle\) is a numerical semigroup.
 Consequently, if \(\gcd(\leftsoper(S))\ne 1\) (equivalently, \(\langle\leftsoper(S)\rangle\) is not a numerical semigroup), then \(S\) has infinitely many descendants.

On the other hand, since any descendant of \(S\) contains \(\langle\leftsoper(S)\rangle\) (by Lemma~\ref{lemma:descendants_contain_left_elements}), the number of descendants of \(S\) can not exceed the number of numerical semigroups containing \(\langle\leftsoper(S)\rangle\), which is finite if \(\gcd(\leftsoper(S)) = 1\). We have thus proved the following result.

\begin{proposition}[{\cite[Theorem~10]{Bras-AmorosBulygin2009SF-Towards}}]\label{prop:infinitely_many_descendants}
	Let \(S\) be a numerical semigroup. Then \[\gcd(\leftsoper(S))\ne 1 \text{ if and only if } S \text{ has infinitely many descendants}.\]
\end{proposition}

Suppose that \(\gcd(\leftsoper(S)) = 1\). Suppose in addition that \(\mathcal{H}\ne \emptyset\). Note that the Frobenius number of \(\langle\leftsoper(S)\rangle\) is \(\max(\mathcal{H})\); denote it by \(f\). 
The semigroup \(\langle\leftsoper(S)\rangle_{f}\in \Delta(S)\) is a descendant of~\(S\). Thus, its child \(\langle\leftsoper(S)\rangle_{f}\setminus\{f\}\) is a descendant of \(S\). Observing that this child is precisely \(\langle\leftsoper(S)\rangle\), we can write the following consequence of the arguments used to prove the proposition above.

\begin{corollary}\label{cor:lefts_descends_from_S_bound_for_genera_and_conductors}
	If \(\gcd(\leftsoper(S))=1\), then  \(\langle\leftsoper(S) \rangle\preceq S\). Furthermore, the genus and the conductor of \(\langle\leftsoper(S) \rangle\) are, respectively, the maximum among the genera and the maximum among the conductors of the numerical semigroups \(S'\) such that \(S'\preceq S\).	
\end{corollary}
As there are only finitely many numerical semigroups with a given conductor or a given genus, one immediately gets the following corollary. 

\begin{corollary}\label{cor:left_descendant_no_bounds}
	Let \(S\) be a numerical semigroup such that \(\gcd(\leftsoper(S)) \ne 1\). Then there is no bound for the conductors nor for the genera of the descendants of \(S\).
\end{corollary}

The semigroup \(\langle\leftsoper(S) \rangle\) (numerical or not) has no big primitives. Thus, the following holds:
\begin{remark}\label{rem:lefts_has_no_descendants}
	The numerical semigroup \(\langle\leftsoper(S) \rangle\) has no descendants.
\end{remark}

Next we give another easy remark that is worth to keep in mind.
\begin{remark}\label{rem:leaf_if_equal_lefts}
	Let \(S\) be a numerical semigroup. The following equivalences hold.
	\[S \text{ is a leaf } \Longleftrightarrow S=\langle\leftsoper(S) \rangle\Longleftrightarrow\primitivesoper(S)=\leftprimitivesoper(S).\]
\end{remark}

\begin{proposition}\label{prop:lefts_max_genus_max_cond_min_edim}
	Let \(S\) be a numerical semigroup such that \(\gcd(\leftsoper(S))=1\), and suppose that \(S\) is not a leaf. Let \(S'\) be a descendant of \(S\) and let \(T=\langle\leftsoper(S)\rangle\). If \(S'\ne T\), then
	\begin{enumerate}
		\item \(\genusoper(S')<\genusoper(T)\);
		\item \(\conductoroper(S')<\conductoroper(T)\);
		\item \(\embeddingdimensionoper(S')>\embeddingdimensionoper(T)\).
	\end{enumerate} 
\end{proposition}
\begin{proof}
	\textit{(1)} and \textit{(2)}:
	By Remark~\ref{lemma:descendants_contain_left_elements}, \(T\subseteq S'\), thus \(\genusoper(S')\le\genusoper(T)\) and \(\conductoroper(S')\le\conductoroper(T)\). As equality of the genera or the conductors would force equality of the sets, we conclude that the inequalities have to be strict. 
	
	\textit{(3)}: As \(T\subseteq S'\), Remark~\ref{rem:primitive_in_subsemigroup} implies that \(\primitivesoper(S')\supseteq\primitivesoper(T)\), and consequently, \(\embeddingdimensionoper(S')\ge\embeddingdimensionoper(T)\). The equality \(\embeddingdimensionoper(S')=\embeddingdimensionoper(T)\) would imply \(\primitivesoper(S')=\primitivesoper(T)\), and therefore \(S' = T\). Thus \(\embeddingdimensionoper(S')\ne\embeddingdimensionoper(T)\), and consequently \(\embeddingdimensionoper(S')>\embeddingdimensionoper(T)\).
\end{proof}

%%\newpage
%%%%%%%%%%% properties and challenges %%%%%%%%%%
%\section{Challenges}\label{sec:properties_challenges}
%\input{properties-challenges}
%\newpage
%%%%%%%%%%% trimming global %%%%%%%%%%
\section{Trimming the whole tree}\label{sec:trimming_global}
% Trimming the whole tree
Although the ideas presented here can be easily adapted to other results, only three properties (with parameters) will be considered. These properties, defined in the first subsection, are motivated by the results presented in Section~\ref{sec:motivating}.
The mission of the parameters is to keep the paper actual even if the theorems are proven with bigger parameters. 
From an experimental point of view, it makes sense to code using these parameters.

Cutting semigroups for properties on numerical semigroups are introduced in the second subsection. They are semigroups that satisfy the properties in cause, so as all their descendants do.

The remaining part of this section explores how to trim the semigroups tree to avoid the cutting semigroups.
Rather than treating this aspect in a unified section we made the choice of devoting a separate subsection to each of the three properties considered. By doing so, one can have an exposition avoiding the need to keep in mind results that are not really needed at the time. This can be an advantage in a first reading. One can see disadvantages, one of them being the repetition of the structure of some subsections and of some arguments.

Some experiments shall appear in~\cite{Delgado2019-Probing}.
 %%%%%%%
\subsection{Some properties}\label{subsec:some_properties}
 	A generic property on numerical semigroups will be denoted by \(\mathcal{P}\).
 	Occasionally we will denote the fact that a numerical semigroup \(S\) satisfies the property~\(\mathcal{P}\) by \(S\models \mathcal{P}\). Next we define the three properties (with parameters) considered in this text.
 
 	\begin{definition}\label{def:property_G}
 		Let \(S\) be a numerical semigroup and let \(g\) be a positive integer. We say that \(S\) has the property \(\prop Gg\) if its genus is no greater than \(g\), that is, \(S\models \prop Gg\Leftrightarrow\genusoper(S) \le g\). %\commentred{<?}
 	\end{definition}
 
 	The following definition is motivated by Theorem~\ref{th:large-mult-are-wilf}. 
 	\begin{definition}\label{def:property_H}
 		Let \(\ell\) be a positive real number. We say that \(S\) satisfies \(\prop H{\ell}\) if \(\conductoroper(S)\le \ell\multiplicityoper(S)\). In symbols, \(S\models \prop H{\ell}\Leftrightarrow \frac{\conductoroper(S)}{\multiplicityoper(S)}\le\ell\).
 	\end{definition}
 	Note that semigroups satisfying \(\prop H{3}\) are the generic ones.
 	We will occasionally refer to semigroups that satisfy \(\prop H{\ell}\) as \emph{\(\ell\)-depth semigroups} or, more vaguely, when \(\ell\) is understood, as \emph{small depth semigroups}. We refer to numerical semigroups that are not of small depth (that is, satisfy \(\frac{\conductoroper(S)}{\multiplicityoper(S)} > \ell\)) as being of \emph{large depth}. 
 
 	Next we present a definition motivated by Theorem~\ref{th:large-ed-are-wilf}.
 	\begin{definition}\label{def:property_D}
 		Let \(\kappa\) be a positive real number. We say that \(S\) satisfies \(\prop D{\kappa}\) if \(\embeddingdimensionoper(S) \ge \multiplicityoper(S)/\kappa\). In symbols, \(S\models \prop D{\kappa}\Leftrightarrow \frac{\multiplicityoper(S)}{\embeddingdimensionoper(S)}\le \kappa\).
 	\end{definition}
 	Recall that the embedding dimension of a numerical semigroup is no bigger than its multiplicity.
 	Numerical semigroups satisfying \(\prop D{\kappa}\) will sometimes be referred to as \emph{\(\kappa\)-large density semigroups} or simply as \emph{semigroups of large density}. Semigroups whose density is not large (that is, satisfying \(\frac{\multiplicityoper(S)}{\embeddingdimensionoper(S)} > \kappa\)) are said to be of \emph{little density}.
	\medskip
 
 	One can obtain new properties by using conjunctions of the properties just defined. The next example illustrates possible uses of the properties defined and some of its conjunctions. 
 	\begin{example}
 		\begin{enumerate}
 			\item 
 			Non Wilf semigroups (that is, counter-examples for Wilf's conjecture) are among the semigroups that do not satisfy the property ``\(\prop D3\) and \(\prop H3\) and \(\mathcal{G}_{60}\)''. Note the use of Theorems~\ref{th:large-ed-are-wilf} and~\ref{th:large-mult-are-wilf} and of Proposition~\ref{prop:genus-60-are-Wilf}.
 			Thus, when searching for non Wilf semigroups, semigroups that are known in advance to satisfy any of these properties do not need to be tested.
 			\item
 			When looking for non Eliahou numerical semigroups or for \(0\)-Wilf semigroups, semigroups known in advance to be \(\prop H3\) semigroups do not need to be considered (so as the \(\mathcal{G}_{60}\)-cutting semigroups, as long as the examples in Example~\ref{ex:Eliahou-Fromentin-examples} are not forgotten). Note the use of Theorem~\ref{th:eliahou-large-mult} and Proposition~\ref{prop:generic_are_0_Wilf}.
 		\end{enumerate}
 	\end{example}
  
%%%%%%%cutting semigroups
\subsection{Cutting semigroups and trimmed trees}\label{subsec:cutting_semigroups}
 	Cutting semigroups will be used to trim the classical tree of numerical semigroups of those branches that are guaranteed not to contribute for the problem under consideration. Note that when a node is cut off, all its descendants are cut off as well -- this gives an indication on how to define cutting semigroups: there must be some kind of heredity.
 
 	\begin{definition}\label{def:cutting_semigroup}
 		Let \(\mathcal{P}\) be a property on numerical semigroups. A numerical semigroup \(S\) is said to be a  \emph{cutting semigroup for \(\mathcal{P}\)} if itself and all its descendants satisfy the property \(\mathcal{P}\). In symbols: \(S\) is a  \emph{cutting semigroup for \(\mathcal{P}\)} if
 		\[S'\preceq S \Rightarrow S'\models \mathcal{P}.\]
 	\end{definition}
 	In other words, a numerical semigroup \(S\) is a cutting semigroup for a property \(\mathcal{P}\) on numerical semigroups if \(S\) satisfies \(\mathcal{P}\) and all the descendants inherit this property from \(S\).
 	The terminology \emph{\(\mathcal{P}\)-cutting semigroup} will also be used. 
 
  	The following remark follows immediately from the definition. 
 	\begin{remark}\label{rem:leaf_cutting}
 		A leaf is \(\mathcal{P}\)-cutting semigroup if and only if it satisfies the property~\(\mathcal{P}\). 
 	\end{remark}
 	Having at hand the tree represented in Figure~\ref{fig:58111214} is perhaps advantageous to better follow the next concrete example, which aims to contribute to better understanding the concept. 
 	\begin{example}
 		Suppose that one is interested in computing the set of semigroups of genus~\(11\) (or simply its cardinality).
 		While exploring the numerical semigroups tree, the node labelled \(S=\langle 5,8,11,12,14\rangle\) is reached at some point, since \(S\) has genus \(7\). This last assertion can be easily verified by hand and confirmed by using \textsf{GAP}, as in the following \textsf{GAP} session.
 		\begin{verbatim}
 			gap> s := NumericalSemigroup(5,8,11,12,14);;
 			gap> Genus(s);
 			7
 		\end{verbatim} \vspace{-\bigskipamount}
 		The numerical semigroups of genus \(11\) in the subtree rooted at \(S\) are  \(\langle 5,8,19,22\rangle\) and \(\langle 5,8,17\rangle\).
 		The semigroup \(S_1=\langle 5,8,11,14,17\rangle\) is among the descendants of \(S\) (in fact, \(S_1=S\setminus\{12\}\)). But the genera of the descendants of \(S_1\) are smaller than \(11\), since the maximum is attained at the numerical semigroup generated by the (coprime) left primitives \(5,8,11\) of \(S_1\), as can be seen analysing Figure~\ref{fig:58111214}. Note that this is just a consequence of Proposition~\ref{prop:infinitely_many_descendants}.  
 		\begin{verbatim}
 		gap> Genus(NumericalSemigroup(5,8,11));
 		10
 		\end{verbatim}  \vspace{-\bigskipamount}
 		As the genus of \(\langle 5,8,11\rangle\) is \(10\), 
 		there is no need to consider neither  \(S_1\) nor its descendants to determine the numerical semigroups of genus \(11\). \(S_1\) is a cutting semigroup for the property \(\mathcal{G}_{10}\): ``has genus smaller than \(11\)''. 
 	\end{example}
 
 	The idea presented in the above example will be formalized below, in Section~\ref{subsec:avoid_small_genera}.
 
 	\begin{definition}\label{def:trimmed-tree}
 		Let \(\mathcal{P}\) be a property on numerical semigroups. The \(\mathcal{P}\)-trimmed tree is the subtree of the classical tree of numerical semigroups consisting of the numerical semigroups that are not \(\mathcal{P}\)-cutting.
 	\end{definition}

	On the absence of possible confusion concerning the property in cause or when referring to a generic property, when referring \(\mathcal{P}\)-cutting semigroups or the \(\mathcal{P}\)-trimmed tree, the \(\mathcal{P}\) may be omitted.
	\medskip

	The concepts of trimmed tree and cutting semigroups may be used with various purposes, as will be clear from what remains in the present text. For instance, observing that the \(\mathcal{P}\)-trimmed tree contains all numerical semigroups that do not satisfy \(\mathcal{P}\) (it also contains those numerical semigroups that have some descendant that does not satisfy \(\mathcal{P}\)), it can be used to count the non-\(\mathcal{P}\)-cutting semigroups.

 	Another, possibly more, inspiring example follows.
 	Suppose that one is interested in finding numerical semigroups that do not satisfy a property~\(\mathcal{Q}\). 
 	If one takes~\(\mathcal{Q}\) as the property ``is Wilf'', the aim is finding a counter example for Wilf's conjecture. Then, assuming that the search is done by exploring the classical tree of numerical semigroups, one can cut the nodes corresponding to cutting semigroups for some property \(\mathcal{P}\) that implies \(\mathcal{Q}\), without loosing any counter-example. Continuing the inspiring example: taking~\(\mathcal{Q}\) as the property ``is Wilf'', one can, for instance, take  \(\mathcal{P}\) as ``has \(3\)-large density'' (see Theorem~\ref{th:large-ed-are-wilf}) or is generic
 	(see Theorem~\ref{th:large-mult-are-wilf}).
  
%%%%%%%
\subsection{Avoiding small genera}\label{subsec:avoid_small_genera}

	Recall that, for a positive integer \(g\), we say that a numerical semigroup \(S\) satisfies the property \(\prop Gg\) if \(\genusoper(S) \le g\).
	From the definition of cutting semigroup for some property it follows that a numerical semigroup is a cutting semigroup for \(\prop Gg\) if and only if neither its genus nor the genus of any of its descendants is greater than \(g\).

	The following proposition characterizes \(\mathcal{G}_g\)-cutting semigroups in terms of the semigroups generated by their left elements.

	\begin{proposition}\label{prop:characterizing-G-cutting-via-left-elements}
		Let \(S\) be a numerical semigroup and let \(T=\langle\leftsoper(S)\rangle\) be the semigroup generated by its left elements. 
		Then \(S\) is a \(\prop Gg\)-cutting semigroup if and only if \(T\) is a numerical semigroup and \(\genusoper(T)\le g\).	
	\end{proposition}
	\begin{proof}
		If \(S\) is a \(\prop Gg\)-cutting semigroup, then \(T\) must be a numerical semigroup, since otherwise \(S\) has descendants with genus greater than \(g\), by Corollary~\ref{cor:left_descendant_no_bounds}. 
		But then, by Proposition~\ref{prop:descendants_correspondance_with_lefts}, \(T\) is a descendant of \(S\). From the definition of \(\prop Gg\)-cutting semigroup, it follows that \(T\) satisfies \(\prop Gg\), that is, \(\genusoper(T)\le g\).
	\medskip
	
		For the converse, assuming that \(T\) is a numerical semigroup such that \(\genusoper(T)\le g\), we prove that if \(S'\preceq S\), then \(\genusoper(S')\le g\).
		If \(S'\) is a descendant of \(S\), then \(T\subseteq S'\), by Lemma~\ref{lemma:descendants_contain_left_elements}. This implies that \(\genusoper(S')\le \genusoper(T)\le g\).
	\end{proof}

	The property \(\mathcal{G}_g\) permits to take advantage of computations formerly done.
	For instance, when looking for the number of numerical semigroups of some genus bigger than \(70\) (notice the connection with Challenge~\ref{challenge:genus}), \(\mathcal{G}_{70}\)-cutting semigroups can be cut off. This, allied to the fact that the number of numerical semigroups is known for genera non greater that \(70\) (Remark~\ref{rem:N(70)}), reduces the number of computations needed. 

	Whether this theoretical approach leads to computational advantages depends on the capability of checking if a numerical semigroup is \(\prop Gg\)-cutting with low computational cost. Notice that Proposition~\ref{prop:characterizing-G-cutting-via-left-elements} reduces the problem to computing a greatest common divisor of the left elements and, if it is \(1\), checking whether the genus of the numerical semigroup generated by the left elements does not exceed a certain number. Although possibly faster than computing the genus, this may be a highly time consuming task (see Section~\ref{sec:technical_details} for some further comments). 
	We do not know how to avoid these extra computations if one aims to trim the tree for the property \(\prop Gg\). 

%%%%%%%%%%
\subsection{Avoiding generic semigroups}\label{sec:avoid_generic_semigroups}
	Recall that, for a positive real number \(\ell\), \(S\) is said to satisfy \(\prop H{\ell}\) if \(\conductoroper(S)\le {\ell}\multiplicityoper(S) \). Generic semigroups are the \(3\)-little depth ones.

	It follows from the definition that a numerical semigroup~\(S\) is a cutting semigroup for~\(\prop H{\ell}\) if and only if neither the conductor of~\(S\) nor the conductor of any of its descendants with the same multiplicity is bigger than \({\ell}\multiplicityoper(S)\).

	The following proposition characterizes the \(\prop H{\ell}\)-cutting semigroups in terms of the semigroups generated by their left elements. Its statement and proof are analogous to those of Proposition~\ref{prop:characterizing-G-cutting-via-left-elements}.
	\begin{proposition}\label{prop:characterizing-M-cutting-via-left-elements}
		Let \(S\) be a numerical semigroup and let \(T=\langle\leftsoper(S)\rangle\).
		Then \(S\) is a \(\prop H{\ell}\)-cutting semigroup if and only if \(\gcd(\leftsoper(S)) = 1\) and \(T\) is a \(\prop H{\ell}\)-cutting semigroup.
	\end{proposition}
	\begin{proof}
		If \(S\) is a \(\prop H{\ell}\)-cutting semigroup, then \(T\) must be a numerical semigroup, since otherwise \(S\) would have descendants with conductor larger than \(\ell\multiplicityoper\) , by Corollary~\ref{cor:left_descendant_no_bounds}. But then, by Corollary~\ref{cor:lefts_is_descendant}, \(T\) is a descendant of \(S\) and therefore is a \(\prop H{\ell}\)-cutting semigroup.
	
    	Conversely, assume that \(\gcd(\leftsoper(S)) = 1\) and that \(T\) is a \(\prop H{\ell}\)-cutting semigroup. This implies that \(\conductoroper(T)\le \ell\multiplicityoper(S)\). By Proposition~\ref{prop:lefts_max_genus_max_cond_min_edim}, the conductor of any descendant of \(S\) is no bigger than \(\conductoroper(T)\). Thus \(S\) is \(\prop H{\ell}\)-cutting, by definition.
	\end{proof}

	Similar comments to the ones made at the end of Section~\ref{subsec:avoid_small_genera}, concerning the property \(\prop{G}{g}\), can be made:
	whether this theoretical approach leads to computational advantages depends on the capability of checking if a numerical semigroup is \(\prop H{\ell}\)-cutting with low computational cost. 
	Notice that Proposition~\ref{prop:characterizing-M-cutting-via-left-elements} reduces the problem of determining \(\prop{H}{\ell}\) to computing a greatest common divisor and checking whether the conductor of a semigroup does not exceed a certain bound. Computing the conductor is well known to be hard (see~\cite{Ramirez-Alfonsin2005Book-Diophantine}).
	Although possibly faster than computing the conductor, testing whether the conductor does not exceed a certain bound can in practice be highly time consuming. Some further comments are in Section~\ref{sec:technical_details}.

	\medskip

	As the \(\prop H{3}\)-trimmed tree contains all non-generic semigroups, trimming the classical tree of numerical semigroups for \(\prop H{3}\) may be used with various purposes related to the challenges presented in Section~\ref{subsec:challenges}:
	\begin{itemize}
		\item count non-generic semigroups (and therefore the generic ones) (Question~\ref{challenge:count_generic});
		\item search for examples of non-Eliahou semigroups (Question~\ref{challenge:eliahou}, see also Theorem~\ref{th:eliahou-large-mult});
		\item search for \(0\)-Wilf semigroups that are non-quasi-superficial and have embedding dimension greater than \(2\) (Question~\ref{challenge:Wilf_number_0}, see also Proposition~\ref{prop:generic_are_0_Wilf}).
		\item test Wilf's conjecture (Question~\ref{challenge:wilf}, see also Theorem~\ref{th:large-mult-are-wilf}) -- Section~\ref{sec:avoid_large_edim} provides a seemingly much more efficient way.
	\end{itemize} 

%%%%%%%%%%
\subsection{Avoiding \(\kappa\)-large density}\label{sec:avoid_large_edim}

	Let \(\kappa>1\) be a real number. Recall that a numerical semigroup \(S\) is said to satisfy \(\prop D{\kappa}\) if \(\embeddingdimensionoper(S) \ge \multiplicityoper(S)/\kappa\). 
	Informally, we refer to it by saying that \(S\) has  \(\kappa\)-large embedding dimension.

	It follows from the definition that a numerical semigroup is a cutting semigroup for \(\prop D{\kappa}\) if and only if neither the embedding dimension of \(S\) nor the embedding dimension of any of its descendants with the same multiplicity is smaller than \(\multiplicityoper(S)/{\kappa}\).

	For a superficial numerical semigroup \(\superficial{m}\) with multiplicity \(m\) we have that \(\embeddingdimensionoper(\superficial{m})=\multiplicityoper(\superficial{m})=m\). As \(\kappa>1\), we have that \(\embeddingdimensionoper(\superficial{m}) > \multiplicityoper(\superficial{m})/\kappa\) and therefore \(\superficial{m}\) satisfies \(\prop D{\kappa}\). 
	Thus we can state the following:

	\begin{proposition}\label{prop:characterizing-superficial-D-cutting-via-left-elements}
		Let \(\kappa>1\) be a real number and \(m\) a positive integer.
		The superficial numerical semigroup \(\superficial{m}\) is not \(\prop D{\kappa}\)-cutting.
	\end{proposition}
	\begin{proof}
		Suppose that \(m\ge\lfloor2\kappa\rfloor+1\). 
		The numerical semigroup \(S=\langle m,m+1\rangle\) is a descendant of \(\superficial{m}\) (it can be obtained from \(\superficial{m}\) by removing \(m+2\) and then successively the smallest big primitive, until there is some). Observe that \(\embeddingdimensionoper(S)=2<\frac{m}{2}=\frac{\multiplicityoper(S)}{2}\). Thus, \(S\) does not satisfy \(\prop D{\kappa}\), and therefore \(\superficial{m}\) is not \(\prop D{\kappa}\)-cutting.
	
		Since a superficial semigroup is ancestor of any superficial semigroup with bigger multiplicity, there is no loss of generality in supposing that \(m\ge\lfloor2\kappa\rfloor+1\), and the result is proved.	
	\end{proof}
	Now that we know that superficial semigroups are not \(\prop D{\kappa}\)-cutting, a full characterization of the \(\prop D{\kappa}\)-cutting semigroups is obtained by treating the case of non-superficial numerical semigroups.

	\begin{proposition}\label{prop:characterizing-non-superficial-D-cutting-via-left-elements}
		Let \(S\) be a non-superficial numerical semigroup and let \(\kappa>1\) be a real number.
		Then:
		\begin{enumerate}[label={\alph*)}]
			\item if \(\gcd(\leftsoper(S)) = 1\), then \(S\) is a \(\prop D{\kappa}\)-cutting semigroup if and only if \(\leftembeddingdimensionoper(S)\ge \frac{\multiplicityoper(S)}{\kappa}\);
			\item if \(\gcd(\leftsoper(S)) \ne 1\), then \(S\) is a \(\prop D{\kappa}\)-cutting semigroup if and only if \(\leftembeddingdimensionoper(S)\ge \frac{\multiplicityoper(S)}{\kappa} -1\).
		\end{enumerate}
	\end{proposition}
	\begin{proof}
		As \(S\) is non-superficial, the multiplicity of any descendant of \(S\) is \(\multiplicityoper(S)\).
	
		If \(S\) is a leaf, then the result follows from Remarks~\ref{rem:leaf_if_equal_lefts} and~\ref{rem:leaf_cutting}.
	
		Assume now that \(S\) is not a leaf. 
	
		\textit{a)}: As \(\gcd(\leftsoper(S)) = 1\) and \(S\) is not a leaf, Corollary~\ref{cor:lefts_descends_from_S_bound_for_genera_and_conductors} guarantees that \(T=\langle\leftsoper(S)\rangle\) descends from \(S\). By 
		Remark~\ref{rem:edim_equals_edim_lefts}, we have \(\embeddingdimensionoper(T)=\leftembeddingdimensionoper(S)\).%\le \embeddingdimensionoper(S)\).
	
		If \(S\) is \(\prop D{\kappa}\)-cutting, then any descendant of \(S\) satisfies \(\prop D{\kappa}\), by definition.
		In particular,  \(T\) satisfies \(\prop D{\kappa}\), and therefore \(\leftembeddingdimensionoper(S)=\embeddingdimensionoper(T)\ge \frac{\multiplicityoper(T)}{\kappa}=\frac{\multiplicityoper(S)}{\kappa}\). 	
		\medskip
	
		Conversely, assume that \(\leftembeddingdimensionoper(S)\ge \frac{\multiplicityoper(S)}{\kappa}\), and let \(S'\) be a descendant of \(S\). By Corollary~\ref{cor:lefts_sequence_increases}, \(\embeddingdimensionoper(S')\ge \leftembeddingdimensionoper(S)\), and therefore \(S'\) satisfies \(\prop D{\kappa}\).
		\medskip
		
		\textit{b)}: As \(\gcd(\leftsoper(S)) \ne 1\), we have that \(\embeddingdimensionoper(S)>\leftembeddingdimensionoper(S)\), since otherwise \(S\) would not be a numerical semigroup.
		Assume that \(\leftembeddingdimensionoper(S)\ge \frac{\multiplicityoper(S)}{\kappa} -1\) (which implies that \(S\) satisfies \(\prop D{\kappa}\)), and let \(S'\) be a descendant of \(S\).
	
		As \(\leftprimitivesoper(S)\subseteq\primitivesoper(S')\), by Corollary~\ref{cor:lefts_sequence_increases}, and the left primitives of \(S\) are not globally coprime, we have that \(\embeddingdimensionoper(S')\ge \leftembeddingdimensionoper(S)+1\).
	
		Thus \(\embeddingdimensionoper(S')\ge \leftembeddingdimensionoper(S)+1\ge\left(\frac{\multiplicityoper(S)}{\kappa} -1\right)+1=\frac{\multiplicityoper(S)}{\kappa}=\frac{\multiplicityoper(S')}{\kappa}\), and \(S'\) satisfies \(\prop D{\kappa}\). As \(S\) also satisfies this property, we conclude that \(S\) is a \(\prop D{\kappa}\)-cutting semigroup.
		\medskip
	
		The converse is proved by contradiction.
		Suppose that \(\leftembeddingdimensionoper(S) < \frac{\multiplicityoper(S)}{\kappa} -1\). A descendant of \(S\) that is not a \(\prop D{\kappa}\)-cutting semigroup will be exhibited.
		Choose a prime \(p\) greater than the conductor of \(S\). Note that, as \(\multiplicityoper(S)<p\), \(\gcd{(\multiplicityoper(S),p)}=1\). Then, by Corollary~\ref{cor:a_descendant}, \(T=\langle \leftsoper(S)\cup\{p\}\rangle\) is a descendant of \(S\) that does not satisfy \(\prop D{\kappa}\), since \(\embeddingdimensionoper(T)=\leftembeddingdimensionoper(S)+1<\frac{\multiplicityoper(S)}{\kappa} -1 +1 = \frac{\multiplicityoper(S)}{\kappa}= \frac{\multiplicityoper(T)}{\kappa}\). Thus \(S\) is not \(\prop D{\kappa}\)-cutting.
	\end{proof}
%gap> SetInfoLevel(InfoViz,1);
%gap> NS := NumericalSemigroup([14..27]);;
%gap> depth := 3;;
%gap> orientation := "rigth";;
%gap> tkz := tikzStringTreeRootedByNumericalSemigroup(NS,depth,[],orientation)_;;_AUTODOC_GLOBAL_OPTION_RECORD  _SubAlgebraModuleHelper
%_InjectionPrincipalFactor      
%gap> tkz := tikzStringTreeRootedByNumericalSemigroup_LaTex(NS,depth,[],orientation);;
%gap> Splash(tkz);
	\noindent\begin{figure}[h]
		%\centering
		\begin{subfigure}[b]{0.063\textwidth}
			\includegraphics[width=\textwidth]{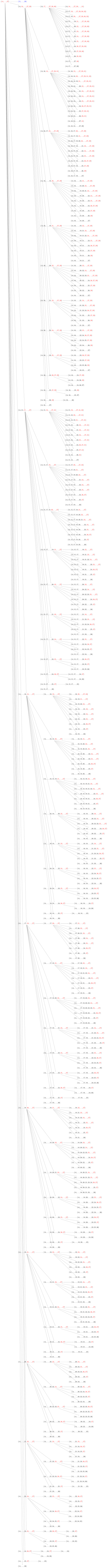}
			\caption{{\(\mathbf{T}_{14}\), \(3\) levels}}
		\end{subfigure}
		\begin{subfigure}[b]{0.85\textwidth}
			\includegraphics[width=\textwidth]{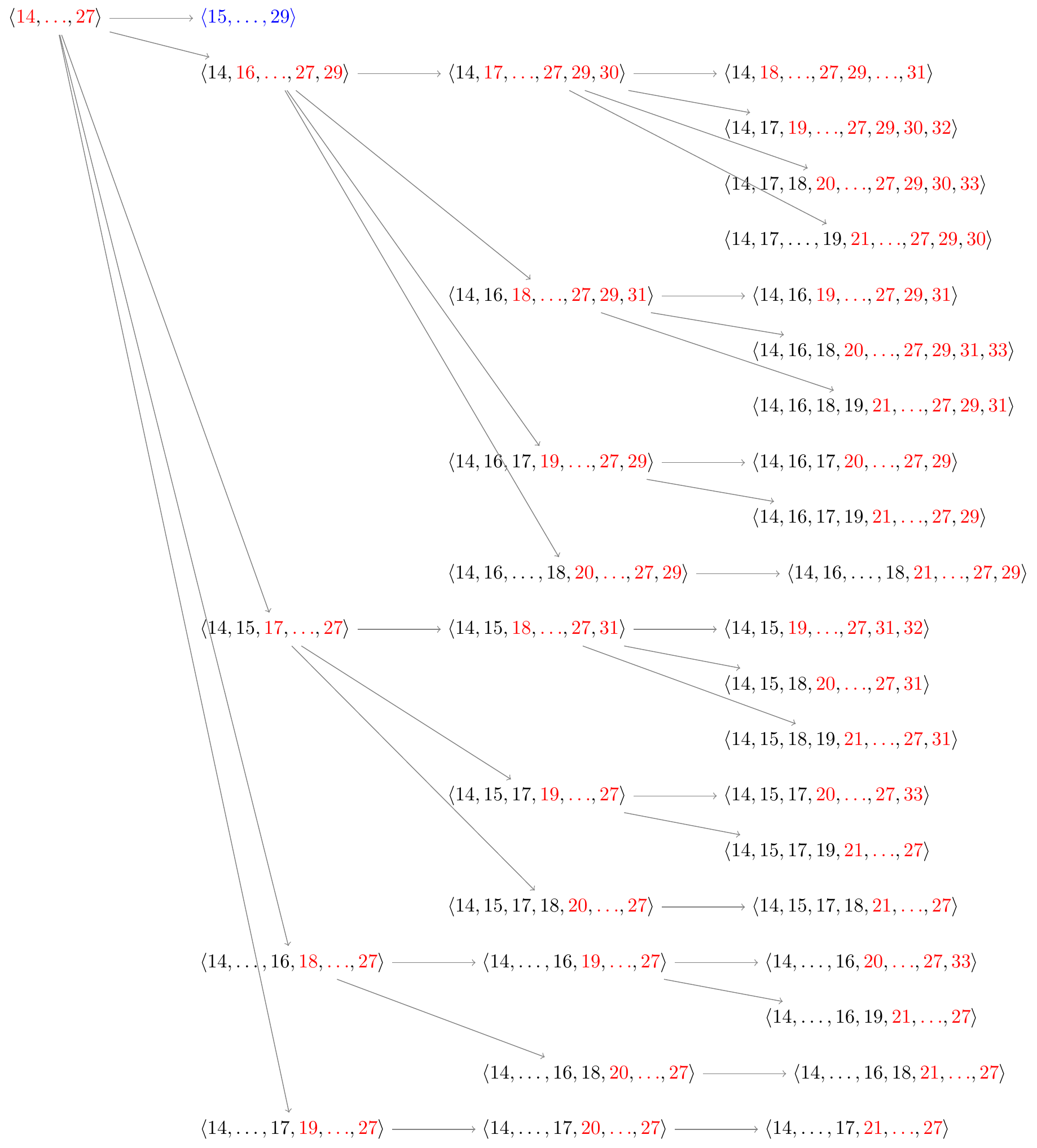}
			\caption{{corresponding trimmed tree}}
		\end{subfigure}
		\caption{Some levels of the trimmed tree rooted \(\mathcal{O}_{14}\). \label{fig:dtrimmed_tree}}
	\end{figure}

	The reader has probably already gained the intuition that trimming for \(\prop D{\kappa}\) is quite efficient in the sense that many semigroups are cut. Although the most probable existence of few non-\(\prop D{3}\) semigroups (as suggested in Section~\ref{subsec:asymptotic_a_result_and_a_problem}) does not lead to infer immediately that many semigroups are cut when trimming the tree, we believe that this intuition is correct.
	Figure~\ref{fig:dtrimmed_tree} helps to have a pictorial indication on this direction. The figure consists of two trees (with the roots drawn to the right). The leftmost shows the (shape of the) first three levels of the tree rooted at \(\mathcal{O}_{14}\) (and considering only the semigroups with multiplicity \(14\), as usual); the picture to the right is the corresponding \(\prop D{3}\)-trimmed tree.
 	
 	The difference in the efficiency occurs in two ways: one is in the sense that many semigroups are cut; the other one is that, contrary to what happens in the two previous sections (where the trimming depends on computing genus or conductors of numerical semigroups), it is easy to determine if a numerical semigroup is \(\prop D{\kappa}\): one just has to compute the cardinality of a known finite set. Figure~\ref{fig:dtrimmed_tree} helps to have a pictorial indication that this intuition is correct. To the left there are drawn (from left to right) the first three levels of the tree rooted at \(\mathcal{O}_{14}\) (and considering only the semigroups with multiplicity \(14\), as usual); the picture to the right is the corresponding \(\prop D{3}\)-trimmed tree.

%%\newpage

%%%%%%%%%%% trimming truncated %%%%%%%%%%
\section{Trimming the truncated tree}\label{sec:trimming_truncated}
%% Trimming the truncated tree
Let \(\Gamma\) be a positive integer, which we can think of as the maximum genus to be considered.
We want to explore a search space of the form \(\mathcal{S}_{\le \Gamma}\), that is, as referred in Section~\ref{subsec:a_search_space}, a finite tree \(\mathbf{T}^{\Gamma}\) obtained from the numerical semigroups tree by truncating at level \(\Gamma\) (which amounts to not considering numerical semigroups with genus greater than \(\Gamma\)). 

The search space \(\mathcal{S}_{\le \Gamma}\) is the one used when counting semigroups by genus. When some property has to be tested one considers as initial search spaces the trimmed trees obtained as in the previous section and then truncate them  at level \(\Gamma\). The notion of cutting semigroup is then adapted to this search space.

Explaining how considering a (trimmed) truncated tree rather than the whole tree can lead to further trimming is the goal of this section.

%%%%%%%%%% 
\subsection{Contribution of truncated trees rooted at superficial semigroups}\label{sec:paralelization}

	In practice, if one wants to explore the classical tree of numerical semigroups, one can run the tree exploring process in parallel.

	Assume that, having counting in mind, one wants to explore the tree up to genus \(\Gamma\), that is, one wants to explore the finite tree \(\mathbf{T}^{\Gamma}\) and count the number of nodes explored. By the construction of the classical tree of numerical semigroups made in Section~\ref{subsec:the_tree}, the subtrees rooted by distinct superficial semigroups have no common nodes, thus one can parallelize the process of exploring the tree by launching one process for each  \(\mathbf{T}_m\). As the root \(\mathcal{O}_m\) of \(\mathbf{T}_m\) has genus \(m-1\), and all the other semigroups in this tree have bigger genus, we may assume the \(m\) is not greater than \(\Gamma+1\).

	With some obvious adaptations, the above arguments on parallelization hold if one wants to test a given property up to a certain genus -- in that case, one should consider trimmed trees instead. 

	We can concentrate on exploring the trees \(\mathbf{T}_m\) (which consists, by construction, of the numerical semigroups of multiplicity \(m\)), and assume that \(m>1\).

	\medskip

	Denote by \(\mathbf{T}_{m}^{\Gamma}\) the finite tree obtained from \(\mathbf{T}_m\) by cutting the semigroups of genus greater than \(\Gamma\).
	The aim of the comments that follow is to help the reader gain some intuition on the shape of \(\mathbf{T}_m^{\Gamma}\).

	Recall that \(\mathbf{T}_m\) is the tree rooted \(\mathcal{O}_m\) whose nodes are the numerical semigroups of multiplicity~\(m\). 
	Since the number of children of \(\mathcal{O}_{m}\) with multiplicity \(m\) is \(m-1\), we have that the number of nodes in the second level of \(\mathbf{T}_m\) is \(m-1\). In particular, this number grows with \(m\). One can then consider the grandchildren and one expects that the number of nodes in the third level continues growing with \(m\), etc. Intuitively, the width of \(\mathbf{T}_m\) grows with \(m\), that is, the larger is the multiplicity, the fatter is the tree.
	The root of \(\mathbf{T}_m\), \(\mathcal{O}_m\), has genus \(m-1\).
	The paths in \(\mathbf{T}_m\) connecting the root and the semigroups of genus~\(\Gamma\) have length \(\Gamma-m\). These are the paths of maximum length in \(\mathbf{T}_{m}^{\Gamma}\).
	Thus, the larger is the multiplicity, the shorter are the paths of maximum length. Intuitively, the larger is the multiplicity, the shorter is \(\mathbf{T}_{m}^{\Gamma}\).
	\medskip

	The trees to be explored go, by letting \(m\) increase, from thin and tall to fat and short. A natural question is: ``What is the contribution of each subtree (that is, of each multiplicity) to the total number of semigroups of a given genus?''.

	Perhaps not surprisingly (after the intuition gained from the above comments), experiments (see~\cite{Delgado2019-Probing}) suggest that the multiplicities that contribute the most are around~\(2\Gamma/3\). Table~\ref{table:counting_by_mult_genus_55} shows the data for genus \(55\).

	\begin{table}[h]
		\begin{tabular}{ r | r| r | r| r| r}
			\hline
			mult.&	&	mult.&	&	mult.&	\\
			\hline
			&  & 19 & \num{825385467} & \num{37} & \num{98547818334}\\
			2 & \num{1} & \num{20} & \num{1435737224} & \num{38} & \num{101846002196}\\
			3 & \num{19} & \num{21} & \num{2076590341} & \num{39} & \num{99563735876}\\
			4 & \num{279} & \num{22} & \num{3307416867} & \num{40} & \num{90170163159}\\
			5 & \num{1522} & \num{23} & \num{4492728653} & \num{41} & \num{73991259070}\\
			6 & \num{11270} & \num{24} & \num{6938843096} & \num{42} & \num{53913734559}\\
			7 & \num{38130} & \num{25} & \num{8894089045} & \num{43} & \num{34318906507}\\
			8 & \num{178158} & \num{26} & \num{12972569483} & \num{44} & \num{18847941433}\\
			9 & \num{497983} & \num{27} & \num{16481758425} & \num{45} & \num{8845246327}\\
			10 & \num{1634613} & \num{28} & \num{22876937724} & \num{46} & \num{3518678913}\\
			11 & \num{3688088} & \num{29} & \num{28376718345} & \num{47} & \num{1177467268}\\
			12 & \num{10656728} & \num{30} & \num{36956478672} & \num{48} & \num{328691328}\\
			13 & \num{20165257} & \num{31} & \num{44716183827} & \num{49} & \num{75769934}\\
			14 & \num{47685160} & \num{32} & \num{54150283092} & \num{50} & \num{14235094}\\
			15 & \num{86795326} & \num{33} & \num{63468305144} & \num{51} & \num{2140914}\\
			16 & \num{177796219} & \num{34} & \num{73270567326} & \num{52} & \num{251279}\\
			17 & \num{286351790} & \num{35} & \num{82874331072} & \num{53} & \num{22154}\\
			18 & \num{549728887} & \num{36} & \num{91678517847} & \num{54} & \num{1379}\\
			&  &  &  & 55 & 54\\
			&  &  &  & 56 & 1\\
			\hline
			Sum & \num{1185229430} & Sum & \num{555793441650} & Sum & \num{585162065779}\\
			\hline
			
	    \end{tabular}
	    \caption{Number of numerical semigroups with genus \(55\) counted by multiplicity \label{table:counting_by_mult_genus_55}}
	\end{table}	
	
	From this table we made Figure~\ref{fig:counting_by_mult_genus_55}, which conveniently illustrates what we said.
	
	\begin{figure}[h]
			\includegraphics[width=1\textwidth]{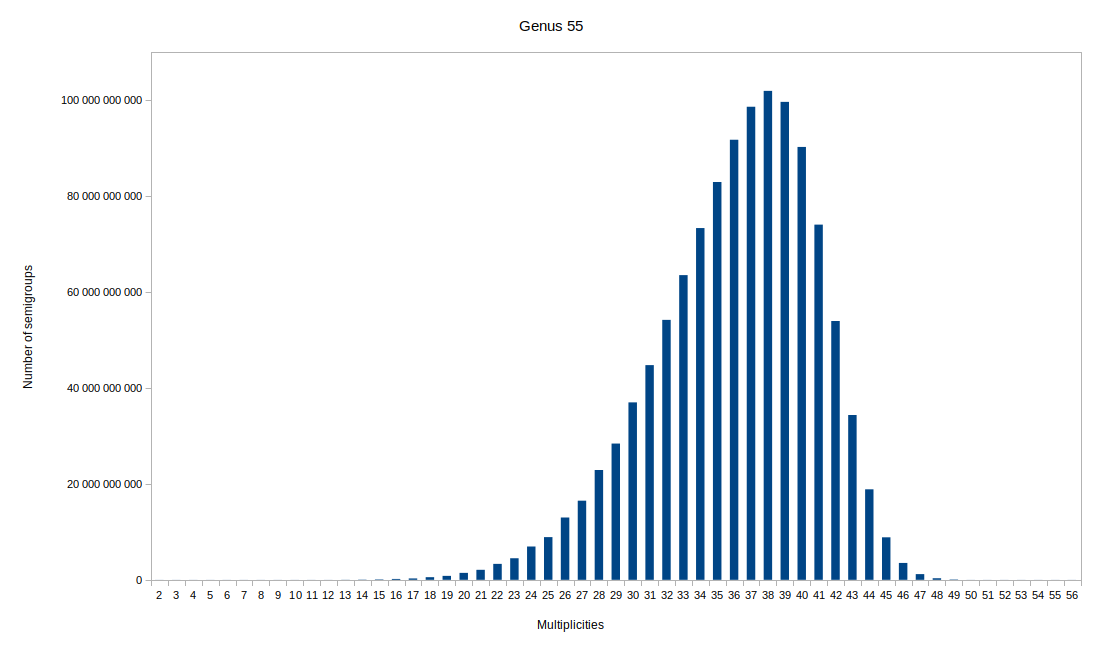}
	\caption{Counting by multiplicity the semigroups of genus  \(55\)}
	\label{fig:counting_by_mult_genus_55}
	\end{figure}

%%%%%%%%
\subsection{Counting by multiplicity and genus}\label{subsec:kaplan}

	Denote by \(N (m, g)\) the number of numerical semigroups with multiplicity \(m\) and genus~\(g\).

	Kaplan proved that when \(m\) and \(g\) are such that \(2g < 3m\) a Fibonacci-like recurrence holds:
	\begin{theorem}[{\cite[Th.~1]{Kaplan2012JPAA-Counting}}]\label{th:kaplan-2/3}
		Let \(m\) and \(g\) be positive integers satisfying \(2g < 3m\).  Then 
		\begin{equation}\label{eq:kaplan-fibonacci}
			N (m, g)=N(m-1, g-1) + N(m-1, g-2).
		\end{equation}
	\end{theorem}
	
	In~\cite{Delgado2019-Probing} one finds a collection of numbers of the form \(N(m,g)\). We can use some to observe that the inequality \(2g < 3m\) in the statement of Theorem~\ref{th:kaplan-2/3} is not replaceable by the non-strict corresponding inequality, which shows that one can not expect a better result in this direction.
	Take \(g=54\) and \(m=2\times 54/3=36\). We have
	\[\begin{array}{lclcl}
	N(m,g)&=&N(36,54)&=&\num{59983372281}\\
	N(m-1, g-1)&=&N(35,53)&=&\num{36397157167}\\
	N(m-1, g-2)&=&N(35,52)&=&\num{23586084042}
	\end{array}\]

	It follows that \(N(35,53)+N(365,52)=\num{59983241209}\ne \num{59983372281}=N(36,54)\).

	\medskip

	Let us now see how to take advantage of Kaplan's result in the process of counting numerical semigroups of a given genus.

	Observe that to compute the number of numerical semigroups of a certain genus \(\Gamma\) by exploring the numerical semigroups tree, all the numerical semigroups of genus lesser than \(\Gamma\) are also explored. The same happens when only numerical semigroups with some fixed multiplicity are considered (which corresponds to explore some \(\mathbf{T}_m\)). Thus, counting the numerical semigroups (with some fixed multiplicity) of genus \(\Gamma\) or up to that same genus requires roughly the same computational effort.

	Suppose that our goal is to compute the number \(N(\Gamma)\) of numerical semigroups of genus~\(\Gamma\), for some positive integer \(\Gamma\).

	Note that if \(m\ge g\), then \(N (m, g)=0\), since for a numerical semigroup \(S\), \(\genusoper(S)\ge \multiplicityoper(S)-1\).
	It follows that \(N(\Gamma)=\sum_{m\ge 2} N(m,\Gamma)=\sum_{2\le m\le \gamma+1} N(m,\Gamma)\). We will split the computation of \(N(\Gamma)\) in two parts: 
	\[\begin{array}{lcl}
		\Sigma_1 & = & \sum_{2\le m \le 2\Gamma/3} N(m,\Gamma)\\
		\Sigma_2 & = & \sum_{2\Gamma/3 < m\le \Gamma+1} N(m,\Gamma)
	\end{array}
	\]

	First we compute \(N(m,g)\) for all integers \(m\) and \(g\) with \(2\le m\le 2\Gamma/3\) and \(1\le g\le \Gamma\). This is a time consuming computation, but, as observed above, can be performed in roughly the same time than just computing \(N(m,\Gamma)\) for the various \(m\). Next we compute \(\Sigma_1=\sum_{2\le m\le 36} N(m,\Gamma)\). Assume that it is done.
		
	Next we will make use of Theorem~\ref{th:kaplan-2/3} to compute \(\Sigma_2\). As \(m > 2\Gamma/3\), we can compute the numbers \(N(\lceil 2\Gamma/3\rceil,g)\), for \(1\le g\le \Gamma\), by computing the sums \(N(\lceil 2\Gamma/3\rceil-1,g-1)+N(\lceil 2\Gamma/3\rceil-1,g-2)\), for each \(g\). (Observe that it is really needed to compute \(N(\lceil 2\Gamma/3\rceil,g)\), for \(1\le g\le \Gamma\), and not just compute \(N(\lceil 2\Gamma/3\rceil,\Gamma)\). For instance, one needs to know \(N(\lceil 2\Gamma/3\rceil,g-1)\) and \(N(\lceil 2\Gamma/3\rceil,g-2)\) to compute \(N(\lceil 2\Gamma/3\rceil+1,g)\).)
	
	Thus, computing \(\Sigma_2\) amounts to compute a number of sums, which is achieved in a negligible amount of time. But, contrary to the time required to compute it, the parcel \(\Sigma_2\) is not negligible in general. That is clearly the case for \(\Gamma=55\):
    
    \begin{example}
    	Let \(\Gamma=55\). Note that \(\lceil 2 \Gamma /3\rceil = 37\) and use the notation \(\Sigma_1\) and \(\Sigma_2\) as above.
		From Table~\ref{table:counting_by_mult_genus_55} we see that \(\Sigma_2= \num{585 162 065 779}\), the last entry of the bottom line of the table. It exceeds the number \(\Sigma_1= \num{556978671080}\), obtained by summing up the two remaining numbers of the same line.
	\end{example}

	The choice of the number \(55\) for the above example is due to the fact that it is the biggest number for which we presently have complete data. And this data suggests that, also for genera bigger than \(55\), about half of the numerical semigroups are counted without time consumption.
	
%%%%%%%
\subsubsection*{Concluding remark}
	Assume that one can count up to genus \(\Gamma\) in a certain prescribed time, without benefiting of the reduction provided by Theorem~\ref{th:kaplan-2/3}. With this reduction  one would be able to count up to genus \(\Gamma+1\) with a non greater computational effort. A naive justification follows: note that \(N(\Gamma+1)\sim 1.61 N(\Gamma)\), but only about half of the \(N(\Gamma+1)\) need to be explored. The remaining half do not need to be explored, and counting them on require some sums, which can be done with almost no time consumption.

%%%%%%%%
\subsection{Counting non-generic semigroups}\label{subsec:counting_non_generic}

	Recall that generic numerical semigroups are those satisfying the property \(\prop H{\ell}\), with \(\ell=3\). Most arguments in this section hold when \(\ell\) is a positive integer.

	Recall that a numerical semigroup is \(\prop H{\ell}\)-cutting if neither itself nor any of its descendants satisfies \(\conductoroper>\ell\multiplicityoper\). Let us adapt this definition to the truncated tree: a numerical semigroup in the \(\prop H{\ell}\)-trimmed tree is \(\prop H{\ell}^{\Gamma}\)-cutting if none of its descendants with genus not greater than~\(\Gamma\) (that is, belonging to the \(\Gamma\)-truncated tree) satisfies \(\conductoroper>\ell\multiplicityoper\). In other words, a numerical semigroup \(S\) such that none of its descendants \(S'\) with \(\genusoper(S')\le \Gamma\) satisfies \(\conductoroper(S')>\ell\multiplicityoper(S)\) can be cut off.
	
	\begin{proposition}\label{prop:big_multiplicity_implies_genericity}
		Let \(S\) be a numerical semigroup such that \(\genusoper(S)\le \Gamma\). If \(\multiplicityoper(S)\ge 2\Gamma/{\ell}\), then \(\conductoroper(S)\le {\ell}\multiplicityoper(S)\).
	\end{proposition}
	\begin{proof}
		By Remark~\ref{rem:bound_for_genus_in_terms_of_conductor} we have that \(\conductoroper(S)\le 2\genusoper(S)\). As, by hypothesis, we have that \(\genusoper(S)\le \Gamma\) we conclude that \(\conductoroper(S)\le 2\Gamma\). It follows that \(\multiplicityoper(S)\ge 2\Gamma/{\ell}\ge \conductoroper(S)/{\ell}\), and therefore \(\conductoroper(S)\le {\ell}\multiplicityoper(S)\).
	\end{proof}
	Taking \(\ell=3\), we get the following:
	\begin{corollary}\label{cor:big_multiplicity_implies_genericity}
		If \(S\) is a non-generic semigroup such that \(\genusoper(S)\le \Gamma\), then \(\multiplicityoper(S)\ge 2\Gamma/{3}\).
	\end{corollary}
	
	So, if we are interested in counting non-generic semigroups up to a given genus, we can parallelize by multiplicity as above and, again, we do not need to consider multiplicities greater than \(2\Gamma/{3}\).

%%%%%%%
\subsubsection{Location of generic / non-generic semigroups}

	Let \(\mathcal O_{\multiplicityoper}\) be the superficial semigroup with multiplicity \(\multiplicityoper\). Let \(S\) be a descendant of \(\mathcal O_{\multiplicityoper}\) with the same multiplicity, that is, \(\multiplicityoper(S)=\multiplicityoper\). As usual, write, respectively, \(\genusoper\) and \(\conductoroper\) for the genus and the conductor of \(S\).
	Recall from Remark~\ref{rem:bound_for_genus_in_terms_of_conductor} that \(\conductoroper/2\le\genusoper\).
	Note that if \(\multiplicityoper\ge 2\), then \(\genusoper\le \conductoroper-1\).
	The following remark says, in particular, that the descendants of \(\mathcal O_{\multiplicityoper}\) that are not known \emph{a priori} to be generic or not are those in the strip of semigroups of genus in the interval \(]3\multiplicityoper/2,3\multiplicityoper[\) (that is, the strip between level \(3\multiplicityoper/2\) and \(3\multiplicityoper\).)
	\begin{remark}\label{rem:generic_and_non-generic_strip}
		In the tree rooted by the superficial semigroup of multiplicity \(\multiplicityoper\ge 2\), all the semigroups up to genus \(3\multiplicityoper/2\) are generic, and all the semigroups with genus non smaller than \(3\multiplicityoper\) are non generic. 
	\end{remark}
	\begin{proof}
		Assume that \(\genusoper \le 3\multiplicityoper/2\). Then \(\conductoroper/2\le \genusoper \le 3\multiplicityoper/2\), which implies that \(\conductoroper\le 3\multiplicityoper\), and therefore \(S\) is generic. If \(\genusoper\ge 3\multiplicityoper\), then \(3\multiplicityoper\le \genusoper < \conductoroper\), thus \(S\) is non generic.
	\end{proof}

	It follows from Remark~\ref{rem:generic_and_non-generic_strip} that semigroups with multiplicity no smaller than \(2/3 \Gamma\) and genus up to \(\Gamma\) are generic. Consequently they are Eliahou and hence Wilf.

%%%%%%%%
\subsubsection*{Concluding remarks}
	Exploring the tree of classical numerical semigroups with the aim of counting numerical semigroups, testing Wilf and finding non Eliahou examples up to genus \(\Gamma\) is reduced to the corresponding computations by multiplicity and genus up to multiplicity \(\lceil 2/3 G\rceil\) and genus \(\Gamma\).
%%%%%%%%
\subsection{Counting semigroups of large density}\label{subsec:counting_large_edim}
	Recall that a numerical semigroup is \(\prop D{\kappa}\)-cutting if neither itself nor any of its descendants satisfies \(\embeddingdimensionoper<\multiplicityoper/\kappa\). As we did in the previous subsection, let us adapt this definition to the truncated tree: a numerical semigroup in the \(\prop D{\kappa}\)-trimmed tree is \(\prop D{\kappa}^{\Gamma}\)-cutting if none of its descendants with genus not greater than~\(\Gamma\) (that is, belonging to the \(\Gamma\)-truncated tree) satisfies \(\embeddingdimensionoper<\multiplicityoper/\kappa\). In other words, a numerical semigroup \(S\) such that none of its descendants \(S'\) with \(\genusoper(S')\le \Gamma\) satisfies \(\embeddingdimensionoper(S')<\multiplicityoper(S)/\kappa\) can be cut off.

	Let \(g\) be a positive integer and let \(e>2\) and \(j\) be positive integers with \(j\le e\le g+1\).
	As a consequence of Corollary~\ref{cor:removing_big_primitive_edim_of_descendant} one has the following:
	\begin{lemma}\label{lemma:edim_of_descendant}
		Let \(g,e,j\) be positive integers such that there exists a numerical semigroup \(S\) with genus \(g\), embedding dimension \(e\) and a descendant with genus \(g+j\). Let \(S'\) be such a descendant of \(S\) of genus \(g+j\). Then \(\embeddingdimensionoper(S')\ge e-j\). Furthermore, every semigroup \(T\ne S'\) in the path from \(S\) to \(S'\) has embedding dimension bigger than \(e-j\).
	\end{lemma}
	
%	Example~\ref{ex:arbitrary_length_sequence_decreasing_edim} (see also
	Remark~\ref{rem:arbitrary_length_sequence_decreasing_edim}) suggests that 
	the integer \(j\) in the previous statement can be arbitrarily large (which forces \(e\) and \(g\) to be even larger). 

	\begin{proposition}\label{prop:cut_superficial_Dk}
		If \(S\) is such that \(\multiplicityoper(S)\ge \frac{\kappa}{2\kappa-1}(\Gamma+1)\), then \(\mathcal{O}_m\) is \(\prop D{\kappa}^{\Gamma}\)-cutting.
	\end{proposition}
	\begin{proof}
		Recall that \(\genusoper(\mathcal{O}_{\multiplicityoper})=\multiplicityoper-1\) and \(\embeddingdimensionoper(\mathcal{O}_m)=\multiplicityoper\). Let \(j=\Gamma-\multiplicityoper+1\), and let 
		%, that is,  \(\Gamma=\multiplicityoper-1+j\). Let 
		\(S'\) be a descendant of \(\mathcal{O}_m\) of genus \(\Gamma=\multiplicityoper-1+j\). By Lemma~\ref{lemma:edim_of_descendant}, \(\embeddingdimensionoper(S')\ge\multiplicityoper-j\), that is,  \(\embeddingdimensionoper(S')\ge 2\multiplicityoper-\Gamma-1\). 
		It follows that if \(2\multiplicityoper-\Gamma-1\ge \multiplicityoper/\kappa\), than any descendant \(S\) of \(\mathcal{O}_m\) up to genus \(\Gamma\) satisfies \(\embeddingdimensionoper(S)\ge\embeddingdimensionoper(S')\ge \multiplicityoper/\kappa\).
		
		%By Theorem~\ref{th:large-ed-are-wilf} we have that if \(2\multiplicityoper-\Gamma-1\ge \multiplicityoper/3\), then all descendants of \(\mathcal{O}_m\) up to genus \(\Gamma\) are Wilf.
		
		As \(2\multiplicityoper-\Gamma-1\ge \multiplicityoper/\kappa\) if and only if \((2\kappa-1)\multiplicityoper \ge \kappa(\Gamma+1)\), the result follows.
	\end{proof}

	By using Theorem~\ref{th:large-ed-are-wilf}, the maximum multiplicity that one needs to consider to test Wilf up to genus \(\Gamma=55\) is \(\multiplicityoper\ge 33 = \left\lceil\frac{3*56}{5}\right\rceil-1\). By considering Figure~\ref{fig:counting_by_mult_genus_55} we see that the proportion of numerical semigroups that do not need to be explored is high. The number of semigroups that do not need to be explored is even higher. In fact, there are no counter-examples to Wilf's conjecture among the semigroups of multiplicity not greater than \(18\), as recently proved by Bruns et al.:
	\begin{proposition}{\cite{BrunsGarcia-SanchezONeilWilburne2019ae-Wilfs}}\label{prop:m18_is_Wilf}
		Every numerical semigroup \(S\) such that \(\multiplicityoper(S)\le 18\) is Wilf. 
	\end{proposition}
	As a consequence, for any \(\Gamma\), the \(\mathbf{T}_{\multiplicityoper}\) with \(\multiplicityoper\le 18\) do not need to be explored. 
	For a given genus, this result does not lead to a big reduction (as suggested by the shape shown in Figure~\ref{fig:counting_by_mult_genus_55}) of the computations needed. But it can be used for every genus, which is great.
	It is used in the following example. Proposition~\ref{prop:cut_superficial_Dk} is used too.

	\begin{example}\label{example:testing-Wilf-reduction}
		Suppose that \(\Gamma=100\). All the \(\mathcal{O}_{\multiplicityoper}\) with \(\multiplicityoper\ge 61 = \left\lceil\frac{3*101}{5}\right\rceil\) can be cut off.
		To test Wilf one only has to consider the \(\mathbf{T}_m\) with \(19\le m\le 60\).
	\end{example}

%%%%%%%%%
\subsection{Further trimming}\label{subsec:further_trimming}
	The above ideas can be extended to trim inner nodes. We limit the discussion here to a simple, but important, case: that of trimming the tree for the property \(\prop{D}{}\). %More general results would be much more tricky.
	
	For any integer \(m>2\), \(\mathbf{T}_m\) can be thought as the disjoint union of some of its subtrees, in the same way we saw the classical tree of numerical semigroups as the disjoint union of subtrees rooted by superficial semigroups. Recall that this is convenient for exploring in parallel. We then observe that, for a fixed \(\Gamma\), some of the subtrees consist of \(\prop{D}{\kappa}^{\Gamma}\)-cutting semigroups and therefore do not need to be explored. 
	\medskip
	
	Let \(m\ge 2\) be an integer. In practice, \(m\) is one of the multiplicities to be considered for computations, following Section~\ref{subsec:counting_large_edim}.
	%s~\ref{subsec:counting_non_generic} and~\ref{subsec:counting_large_edim}. 
	In addition, the associated tree consists of a large number of semigroups (see Section~\ref{sec:paralelization}).

	To a nonnegative integer \(i\) associate the set \(I_i\) obtained from the interval of \(m\) consecutive integers starting in \(m+i+1\) by removing the multiple of \(m\) in that interval, that is, \(I_i=\{x\in \mathbb{N}\cap[m+i+1,2m+i]\mid x\not\equiv 0 \pmod m\}\). Then, consider the numerical semigroup \(S_i=\langle\{m\}\cup I_i\rangle\). Note that \(S_0\) is just the superficial semigroup with multiplicity~\(m\). 
	Next we state a simple remark.
	\begin{remark} Let \(i\ge 1\).
		The smallest big primitive of \(S_i\) is \(m+i+1\). Furthermore,
		\(S_{i+1}=S_i\setminus \{m+i+1\}\) and \(\genusoper(S_i)=m-1+i\).
	\end{remark}

	As a consequence, each semigroup \(S_i\) (\(i\ge 1\)) in the path \(S_0,S_1,\ldots, S_k\) is obtained from the preceding semigroup by removing the leftmost big primitive (which happens to be the smallest primitive besides the multiplicity).  
	Also, if \(i>\Gamma-m+1\), then \(S_i\) does not belong to \(\mathbf{T}^{\Gamma}\).
	\begin{lemma}\label{lemma:edim-Si}
		All the \(S_i\) have embedding dimension \(m\).
	\end{lemma}
	\begin{proof}
		Let \(s\) be a decomposable element of \(S_i\), that is, \(s\) can be written as a linear combination with nonnegative integer coefficients of at least two elements in \(\{m\}\cup I_i\). If the linear combination involves an element of \(I_i\), then \(s\) is bigger than the maximum of \(I_i\) (since \(I_i\) has length \(m\)). Otherwise, \(s\) is a multiple of \(m\) bigger than \(m\). In either case we have that \(s\not\in \{m\}\cup I_i\).
	\end{proof}

	Denote by \(\mathbf{T}_{m,i}\) the subtree rooted at \(S_i\), trimmed of the semigroup \(S_{i+1}\). (Notice the similarity with \(\mathbf{T}_m\), where the child \(\mathcal{O}_{m+1}\) of the root was not considered in the construction of the tree.) Then \(\mathbf{T}_m\) is obtained by adding edges \(\mathcal{S}_i\longrightarrow \mathcal{S}_{i+1}\)
	to the forest formed by the trees \(\mathbf{T}_{m,i}\).
	\medskip
	
	Suppose that we want to explore \(\mathbf{T}_m\) up to genus \(\Gamma\), and trim it for \(\prop{D}{\kappa}^{\Gamma}\). 	
	We proceed as in the proof of Proposition~\ref{prop:cut_superficial_Dk}.
	
	Let \(S'\) be a descendant of \(S_i\) of genus \(\Gamma\). Let \(j = \Gamma - \genusoper(S_i)= \Gamma-m+1-i\).
	By Lemma~\ref{lemma:edim-Si}, \(m(S_i)=m\), and by Lemma~\ref{lemma:edim_of_descendant}, \(\embeddingdimensionoper(S')\ge m-j\).
	It follows that \(\embeddingdimensionoper(S')\ge %m-(\Gamma-m+1-i)=
	2m-\Gamma-1+i\).
	If \(2m-\Gamma-1+i\ge m/3\), then, by Theorem~\ref{th:large-ed-are-wilf}, all descendants of \(S_i\) up to genus \(\Gamma\) are Wilf. Thus we can state the following:
	
	\begin{remark}
		With the above notation, if \(i \ge \Gamma + 1 -\lceil5m/3\rceil\), then \(S_i\) is \(\prop{D}{3}^{\Gamma}\)-cutting.
	\end{remark}
	
	The following example illustrates the consequences for the genus \(100\) case. Various multiplicities are considered.
	\begin{example}
		\begin{description}
			\item [\(m=30\)]
			Note that \(\genusoper(S_{71})=100\). As
			\(101 - \lceil5\times 30 / 3 \rceil = 51\), the semigroup \(S_{51}\) in \(\mathbf{T}_{30}^{100}\) can be cut of. %In particular because \(51\) is quite distant from \(71\), the amount of numerical semigroups that one can avoid to explore due to this fact is considerable. 
			
			Note that if we want to explore \(\mathbf{T}_{30}^{100}\) in parallel by attributing each \(\mathbf{T}_{30,i}^{100}\) to a processor, there is no need to consider the \(\mathbf{T}_{30,i}^{100}\) with \(i\ge 51\).
			
			\item [\(m=40\)]
			Note that \(\genusoper(S_{61})=100\). As
			\(101 - \lceil5\times 40 / 3 \rceil = 101-70 = 31\), thus the semigroup \(S_{31}\) in \(\mathbf{T}_{40}^{100}\) can be cut of.
			
			\item [\(m=50\)]
			Note that \(\genusoper(S_{51})=100\). As
			\(101 - \lceil5\times 40 / 3 \rceil = 101-84 = 17\), thus the semigroup \(S_{17}\) in \(\mathbf{T}_{40}^{100}\) can be cut of.
		\end{description}
	\end{example}

	Adapting the ideas of Section~\ref{subsec:counting_non_generic} to trim \(\mathbf{T}_m\) for the property \(\prop{H}{}\) is equally easy and is left to the reader.

%%\newpage
%%%%%%%%%%% further trimming %%%%%%%%%%
%\section{Further trimming}\label{sec:further_trimming}
%\input{further-trimming}
%\newpage
%%%%%%%%%%% technical details %%%%%%%%%%
\section{A few technical details}\label{sec:technical_details}
% technical
This section aims to give some indication on how to apply in practice the ideas previously developed. It
has some computational flavour, without entering into too many technical details. Implementations will be open source, and part of~\cite{NStree-2019}.

Assume that the (big) primitives are known throughout the computations. This may seem not to be a reasonable assumption, since when removing a big primitive the set of big primitives must be recomputed (a time consuming task), but on the other hand, we found no way to avoid this aspect.
%%%%%%%%%%%
\subsection{On the determination of skilled big primitives}

	Let \(\mathcal{P}\) be a property on numerical semigroups. A \(\mathcal{P}\)-\emph{skilled primitive} of a numerical semigroup \(S\) is a big primitive \(p\in S\) such that \(S\setminus \{p\}\), the semigroup obtained by removing \(p\), is not a \(\mathcal{P}\)-cutting semigroup.

	Skilled primitives are the big primitives that have the skills to produce non-cutting semigroups.
	They serve the purpose of getting rid of the cutting semigroups when the semigroups tree is explored (possibly motivated by the search of counter-examples). Conceptually, these ideas are quite nice, since, when applied, less is left to be explored. But they are only of practical interest if one finds some way of determining the skilled primitives that compares favourably to the complete exploration, in terms of processing time. The remainder of this section deals with the determination of skilled primitives for the properties formerly introduced.

	The following proposition is an immediate consequence of the characterization for the \(\prop D{\kappa}\)-cutting semigroups given by Proposition~\ref{prop:characterizing-non-superficial-D-cutting-via-left-elements}.
	\begin{proposition}\label{prop:skilled-D}
		Let \(S\) be a non-superficial numerical semigroup and let \(\primitivesoper(S)=\{p_1<p_2<\ldots<p_e\}\) be the set of its primitives. 
		The set \(sk(S)\) of \(\prop D{\kappa}\)-skilled primitives of \(S\) is:
		\begin{enumerate}
			\item if \(\gcd(\leftsoper(S)) = 1\), then
			\(sk(S)=\left\{p_i\in \bigprimitivesoper(S) \mid i < \frac{\multiplicityoper(S)}{\kappa}\right\};\)
			\item if \(\gcd(\leftsoper(S)) \ne 1\), then
			\(sk(S)=\left\{p_i\in \bigprimitivesoper(S)\mid i < \frac{\multiplicityoper(S)}{\kappa} -1 \right\}.\)
		\end{enumerate}
	\end{proposition}
	In practice, one can efficiently determine the \(\prop D{\kappa}\)-skilled primitives, since \(\bigprimitivesoper\) is known at each step of the computations.

	One has similar results for the sets of \(\prop H{\ell}\)-skilled primitives and for the \(\mathcal{G}_{g}\)-skilled primitives. These are consequences of Propositions~\ref{prop:characterizing-M-cutting-via-left-elements} and~\ref{prop:characterizing-G-cutting-via-left-elements}, respectively.

	\begin{proposition}\label{prop:skilled-M}
		Let \(sk(S)\) be the set of \(\prop H{\ell}\)-skilled primitives of the numerical semigroup \(S\). The following occurs:
		\begin{enumerate}
			\item if \(\gcd(\leftsoper(S)) = 1\), then
			\(sk(S)=\left\{p\in \bigprimitivesoper(S)\mid \conductoroper(S) \le p \right\};\)
			\item if \(\gcd(\leftsoper(S)) \ne 1\), then all the big primitives of \(S\) are skilled.
		\end{enumerate}
	\end{proposition}

	\begin{proposition}\label{prop:skilled-G}
		Let \(sk(S)\) be the set of \(\mathcal{G}_{g}\)-skilled primitives of the numerical semigroup \(S\). The following occurs:
		\begin{enumerate}
			\item if \(\gcd(\leftsoper(S)) = 1\), then
			\(sk(S)=\left\{p\in \bigprimitivesoper(S)\mid \genusoper(S) \le g \right\};\)
			\item if \(\gcd(\leftsoper(S)) \ne 1\), then all the big primitives of \(S\) are skilled.
		\end{enumerate}
	\end{proposition}
	
	One may immediately suspect that determining the \(\mathcal{H}\)-skilled or the \(\mathcal{G}\)-skilled primitives is much harder than determining the \(\mathcal{D}\)-skilled ones. At first sight one has to determine the conductor or the genus of the semigroup generated by the left elements, which may be big, even when only small primitives (recall the case of semigroups generated by \(2\) elements: the conductor approaches the product of the generators (by Sylvester's formula))! Fortunately, the determination of the conductor or of the genus is not really needed and the results below can be used. Check also Example~\ref{example:check_small_conductor}.

	Computing the conductor of a  numerical semigroup given through a set of generators is well known to be hard (see~\cite{Ramirez-Alfonsin2005Book-Diophantine}). 
	When a good upper bound is known in advance, the problem turns out to be less time consuming. When it is the case, Corollary~\ref{cor:conductor-bounded} can be applied. 

	\begin{lemma}\label{lemma:conductor_truncated_sgp}
		Let \(S\) be a numerical semigroup with multiplicity \(\multiplicityoper\), and let \(k\) be a positive integer. Then
		\begin{equation*}
		\conductoroper(S)\le k \text{ if and only if } \conductoroper(\langle S\rangle_{k+\multiplicityoper -1})\le k.
		\end{equation*} 
	\end{lemma}
	\begin{proof}
		Clearly \(\conductoroper(\langle S\rangle_{k+\multiplicityoper -1})\le\conductoroper(S)\).
		Thus, if \(\conductoroper(S)\le k\), then \(\conductoroper(\langle S\rangle_{k+\multiplicityoper -1})\le k\).
	
		In order to prove the converse, suppose that \(\conductoroper(\langle S\rangle_{k+\multiplicityoper-1})\le k\). Then all the integers \(t\) such that \(k\le t< k+m\) belong to \(S\), which forces all the integers non smaller than \(k\) to belong to \(S\). 
	\end{proof}
	\begin{corollary}\label{cor:conductor-bounded}
		If \(\conductoroper(\langle S\rangle_{k+\multiplicityoper -1})\le k\), then \(\conductoroper(S)=\conductoroper(\langle S\rangle_{k+\multiplicityoper -1})\). 
	\end{corollary}
	When one just wants to know whether the conductor is smaller than a certain given quantity, one may profit of Lemma~\ref{lemma:conductor_truncated_sgp} to get a speed up, by proceeding as illustrated in the following example.

	\begin{example}\label{example:check_small_conductor}
		Let \(S=\langle G\rangle\) be a numerical semigroup given through a set \(G\) of generators.
		In order to test whether the  semigroup \(\langle G\rangle\) satisfies \(\conductoroper(\langle G\rangle)\le 3\multiplicityoper\) (that is, to check if \(S\) is generic), it suffices to check whether \(\conductoroper(\langle G\rangle_{4\multiplicityoper})\le 3\multiplicityoper\). This amounts to verify whether \(\{t\mid 3\multiplicityoper\le t <4\multiplicityoper\}\) consists of elements in \(G+G+G+G\) (that is, of elements with factorization lengths no bigger than \(4\)).
	\end{example}

	Computing the genus of a numerical semigroup from a set of generators is hard (as is computing the Frobenius number). For our purposes, a similar strategy to the above one for the conductor (and making use of it too) may be useful.
	\begin{lemma}
		Let \(S\) be a numerical semigroup with multiplicity \(\multiplicityoper\), let \(g\) be a positive integer and let \(T =\langle S\rangle_{2g+m}\). Then
		\[\genusoper(S)\le g \text{ if and only if }\genusoper(T)\le g.\]
	\end{lemma}
	\begin{proof}
		From \(S\subseteq T\) it follows that  \(\genusoper(S)\ge \genusoper(T)\). It follows that if \(\genusoper(S)\le g\) then \(\genusoper(T)\le g\).
		
		For the converse, note that if \(\conductoroper(S)\le 2g+\multiplicityoper\), then \(S=T\) and nothing has to be proved in this case. If \(\conductoroper(S)>2g+\multiplicityoper\), then \(\conductoroper(T)> 2g\). But then \(\genusoper(T)>g\), by Remark~\ref{rem:bound_for_genus_in_terms_of_conductor}, and consequently \(\genusoper(S)>g\). 
	\end{proof}

	The following remark is just meant to illustrate a kind of results that one can obtain to speed up computations. It says that in numerical semigroups having some very large primitive all big primitives are \(\prop H{\ell}\)-skilled.
	
	\begin{remark}\label{rem:large_primitive_implies_no_Cl_cutting}
		Let \(S\) be a non-superficial numerical semigroup and let \(g=\max(\primitivesoper(S))\). If \(g>(\ell+1)\multiplicityoper(S)\), then no descendant of \(S\) is a \(\prop H{\ell}\)-cutting semigroup. 
	\end{remark}
	\begin{proof}
		Let \(S'\) be a numerical semigroup such that \(S'\preceq S\).
		Note that, as \(S\) is non-superficial, \(\multiplicityoper(S')=\multiplicityoper(S)\).
		If \(\conductoroper(S')>\ell\multiplicityoper(S)\), then \(S'\) is non \(\prop H{\ell}\)-cutting. If \(\conductoroper(S')\le\ell\multiplicityoper(S)\), then \(g\) is a big primitive of \(S'\), and therefore \(S''=S'\setminus\{g\}\) is a descendant of \(S'\) (and of \(S\)). As \(\conductoroper(S'')>g>(\ell+1)\multiplicityoper(S)\), we have that \(S''\) does not satisfy \(\prop H{\ell}\) and therefore \(S'\) is not \(\prop H{\ell}\)-cutting.
	\end{proof}

%%%%%%%
\subsection{Encoding with redundancy}\label{sec:encoding_with_redundancy}
	To compute with numerical semigroups one needs to encode them in some way. 
	Fromentin and Hivert~\cite{FromentinHivert2016MC-Exploring} did the space complexity analysis for the case of the encoding they used (which had some redundancy).  
	For any reasonable encoding, the memory used to explore in a depth first manner the tree of classical numerical semigroups (up to a given level) stays beyond a not too high bound. 
	As a consequence, one can use some encoding with redundancy without entering into memory problems as the tree is being explored. 
	Although the primitives suffice to encode the semigroup, one may wonder on why not keep storing other data that can be computed at a low computational cost (which is redundant, in what concerns the description of the semigroup, but can be used to speed up computations)? For instance, storing the number of left elements, the number of primitives and the conductor can be done at a very low computational cost. This data can be used to compute the Wilf number. Other data, for instance to speed up the test of whether the new primitive candidate is in fact primitive, can also be stored.

	Thus, one can count, test Wilf, \(0\)-Wilf and Eliahou at a cost that is not much higher than just counting. Our experiments suggest that all the referred tests can be done using this strategy in about three times the time required for just counting.
	\medskip

	Now we are in a position to summarize and better justify the great efficiency one finds when dealing with the property \(\prop D{\kappa}\), opposed to the one found for the other properties, as pointed out along the present text.

	The left embedding dimension can be computed at a low computational cost from the data mentioned above. Thus, trimming the tree for \(\prop D{\kappa}\) is not costly. And the number of semigroups is drastically reduced, as suggested by Table~\ref{table:by_genus_low_edim} (see also Figure~\ref{fig:dtrimmed_tree}).

	There is a major drawback in what concerns the other properties considered.
	I have not been able to efficiently do some computations involving the semigroups generated by left elements. It is the case of checking whether the conductor or the genus of the semigroup generated by the left elements stay below certain bounds, for instance.

\newpage

%%%%%%%%%%% Acknowledgements %%%%%%%%
\subsubsection*{Acknowledgements}\label{sec:acknowledgements}
I would like to thank my colleges at the FCUP's Mathematics department who made possible for me to benefit of a sabbatical during the academic year 2018/2019. The hospitality found in the Instituto de Matemáticas de la Universidad de Granada (IEMath-GR), so as at the Laboratoire de Mathématiques Pures et Appliquées Joseph Liouville (LMPA), a research centre of the l’Université du Littoral Côte d’Opale, was amazing. This made the writing of this paper a lot easier. 
Many thanks to the IEMath-GR and the LMPA. These stays in Granada and Calais were made possible by Pedro García-Sánchez and Shalom Eliahou, respectively, and I would like to thank them both.

I would also like to warmly thank many people that patiently heard my ideas and motivations, in particular those organizing or attending the regular \emph{International Meeting on Numerical Semigroups}. 
Although risking some unforgivable forgetfulness, I would like to mention among them, in alphabetical order: Claude, Jean, Pedro and Shalom. Their surnames are Marion, Fromentin, García-Sánchez and Eliahou, respectively. I highly appreciated their feedback along the various phases of this project.

%%%%%%%%%%% Acknowledgements %%%%%%%%
% in order to use bibtex comment the following line and uncomment the others
%% uncomment to use biblatex  
%\printbibliography
%%
%%uncomment to use bibtex
% \bibliographystyle{plainurl}
% \bibliography{../../../../../bib/NumericalS-bib/numericals.bib,../../../../../bib/NumericalS-bib/preprints.bib,../../../../../bib/NumericalS-bib/software.bib,../../../../../bib/NumericalS-bib/slides.bib}
%%%%%%%%%

\end{document}